\documentclass[12pt]{amsart}
\usepackage[english]{babel}
\usepackage{amsmath,amssymb, amsthm,amsbsy} 

\usepackage{graphicx}
\usepackage{geometry}
\usepackage{hyperref}  
\hypersetup{colorlinks=true, linkcolor=blue, anchorcolor=blue, 
citecolor=red, filecolor=blue, menucolor=blue,
urlcolor=blue}

\numberwithin{equation}{section}
\newtheorem{theorem}{Theorem}[section]
\newtheorem{corollary}[theorem]{Corollary}
\newtheorem{proposition}[theorem]{Proposition}
\newtheorem{lemma}[theorem]{Lemma}
\theoremstyle{remark}
\newtheorem{remark}{Remark}[section]

\theoremstyle{definition}
\newtheorem{definition}[theorem]{Definition}

\newcommand{\R}{\mathbb{R}}
\newcommand{\C}{\mathbb{C}}
\newcommand{\N}{\mathbb{N}}
\newcommand{\Z}{\mathbb{Z}}
\newcommand{\M}{\mathcal{M}}
\newcommand{\K}{\mathbb{K}}

\newcommand{\lap}{\Delta}

\newcommand{\an}[1]{\langle #1 \rangle}

\def\XXint#1#2#3{{\setbox0=\hbox{$#1{#2#3}{\int}$ }
\vcenter{\hbox{$#2#3$ }}\kern-.58\wd0}}

\begin{document}

\title
[Strichartz on curved space]
{Global Strichartz estimates for the Dirac equation on symmetric spaces}

\begin{abstract}
In this paper we study global-in-time, weighted Strichartz estimates for the Dirac equation on warped product spaces in dimension $n\geq3$. In particular, we prove estimates for the dynamics restricted to eigenspaces of the Dirac operator on the compact spin manifolds defining the ambient manifold under some explicit sufficient condition on the metric,  and estimates with loss of angular derivatives for general initial data in the setting of spherically symmetric and asymptotically flat manifolds.

\end{abstract}

\date{\today}

\author{J. Ben-Artzi}
\address{School of Mathematics,
Cardiff University, Senghennydd Road, Cardiff CF24 4AG, Wales, United Kingdom}
\email{Ben-ArtziJ@cardiff.ac.uk}

\author{F. Cacciafesta}
\address{Dipartimento di Matematica, Universit\'a degli studi di Padova, Via Trieste, 63, 35131 Padova PD, Italy}
\email{cacciafe@math.unipd.it}

\author{A. S. de Suzzoni} 
\address{CMLS, \'Ecole Polytechnique, CNRS, Universit\'e Paris- Saclay, 91128 PALAISEAU Cedex, France}
\email{ anne-sophie.de-suzzoni@polytechnique.edu}

\author{J. Zhang}
\address{Department of Mathematics, Beijing Institute of Technology, Beijing
100081, China}
\email{zhang\_junyong@bit.edu.cn}

\thanks{JBA  acknowledges support from an Engineering and Physical Sciences Research Council Fellowship (EP/N020154/1). FC acknowledges support from the University of Padova STARS project ``Linear and Nonlinear Problems for the Dirac Equation" (LANPDE). ASdS acknowledges support from ESSED ANR-18-CE40-0028. JZ acknowledges support from National Natural
Science Foundation of China (11771041, 11831004) and a Marie Sk\l odowska-Curie Fellowship (790623).}

\maketitle


\section{The radial Dirac Equation on symmetric manifolds}\label{radsec}

In \cite{cacdes1,cacdescurv1} the second and third authors have started the study of the dynamics of the Dirac equation on curved spaces, the natural setting being a $4$-dimensional manifold $(\mathcal{M},g)$ with signature $\{+,-,-,-\}$ that decouples space and time: namely, the metric $g$ is assumed to take the form
\begin{equation}\label{struct}
g_{\mu\nu} = \left \lbrace{\begin{array}{ll}
1& \textrm{ if } \mu= \nu = 0\\
0 & \textrm{ if } \mu\nu = 0 \textrm{ and } \mu\neq \nu\\
-h_{\mu\nu}(\overrightarrow x) & \textrm{ otherwise.}\end{array}} \right.
\end{equation}
We recall that the Cauchy problem for the Dirac equation in this setting can be written as
\begin{equation}\label{diraceqcurv}
\begin{cases}
i \partial_t u - \mathcal{D} u-m\beta u = 0,\\
u(0,x)=u_0(x),
\end{cases}
\end{equation}
where $\beta$ is a square, complex matrix such that $\beta^2$ is the identity, and $\mathcal{D}$ is the Dirac operator. By construction, the operator $\mathcal{D}$ satisfies the following property:
\begin{equation}\label{square}
\mathcal{D}^2=-\Delta_h+\frac14\mathcal{R}_h,
\end{equation}
where $\Delta_h$  is the Laplace-Beltrami operator for Dirac bispinors, that is, $\lap_h = D^jD_j$ where $D_j$ is the covariant derivative for Dirac bispinors that we properly define later, and $\mathcal R_h$ is the scalar curvature associated to the spatial metric $h$.  

In \cite{cacdes1} the authors  exploited the classical Morawetz multiplier technique in order to obtain local smoothing (or weak dispersive) estimates for the solutions to equation \eqref{diraceqcurv} in the setting of asymptotically flat and (some) warped products manifolds. As it is often the case when dealing with equations on manifolds, it is not possible to rely on the classical Duhamel argument in order to obtain Strichartz estimates for the flow, due to the fact that, even in the asymptotically flat case, the perturbative term can not be regarded as a zero-order perturbation of the flat dynamics. 

In the subsequent paper \cite{cacdescurv1}, the authors  considered  $3$-dimensional spherically symmetric settings, that is manifolds $(\mathcal{M},g)$ defined by $\mathcal{M}=\mathbb{R}_t\times \Sigma$ where now $\Sigma=\mathbb{R}^+_r\times\mathbb{S}^2_{\theta,\phi} $ is equipped with the Riemannian metric
\begin{equation}\label{sphermet}
dr^2+\varphi(r)^2d\omega_{\mathbb{S}^2}^2
\end{equation}
where $d\omega_{\mathbb{S}^2}^2=(d\theta^2+\sin^2\theta d\phi^2)$ is the Euclidean metric on the 2D sphere $\mathbb{S}^2$. Notice that taking $\varphi(r)=r$ reduces $\Sigma$ to the standard 3D Euclidean space, and therefore $\mathcal{M}$ to the standard Minkowski space. Within this setting, in \cite{cacdescurv1} local-in-time, weighted Strichartz estimates for the Dirac dynamics were proved, under some quite general (and natural) assumptions on the function $\varphi$, that will be discussed in forthcoming Subsection \ref{assub}: the main strategy consisted in exploiting the spherical symmetry of the space in order to separate variables and to reduce the problem to a ``sum" of much easier radial equations that could be regarded, after introducing weighted bispinors, as Dirac equations on the flat space perturbed with potentials, for which several results are available. Nevertheless, global-in-time Strichartz estimates turned out to be out of reach, the main problem being the lack of existence of dispersive estimates for the Dirac equation with scaling critical potentials in the Euclidean setting.
\medskip

The purpose of this manuscript is to complement the results of \cite{cacdescurv1} investigating the validity of weighted, global-in-time Strichartz estimates in the more general setting of \emph{warped products} in any space dimension $n\geq3$. We consider manifolds $(\mathcal{M},g)$ defined by $\mathcal{M}=\mathbb{R}_t\times \Sigma$ with $\Sigma=\mathbb{R}_r^+\times\mathbb{K}^{n-1}$ where $\mathbb{K}^{n-1}$ is now a generic $n-1$-dimensional compact and Riemannian spin manifold, and $\Sigma$ is a Riemannian manifold equipped with the Riemannian metric
\begin{equation}\label{warped}
dr^2+\varphi(r)^2d\omega_{\mathbb{K}^{n-1}}^2.
\end{equation}
Here, $\varphi$ is a map from $\R_+$ to itself and $d\omega_{\mathbb{K}^{n-1}}^2$ is the Riemannian metric on $\mathbb{K}^{n-1}$. Of course, this case includes the spherically symmetric one when choosing $\mathbb{K}^{n-1}=\mathbb{S}^{n-1}$, and thus this paper can be regarded in fact as an extension of \cite{cacdescurv1}. On the other hand, as we will see, the assumptions on the admissible functions $\varphi$ will be much stronger: this is due to the fact that, as mentioned, we cannot directly rely on the theory of the flat Dirac equation with potentials, but we need to square the equation at the radial level, in order to reduce to a system of Klein-Gordon equations and then, via Kato smoothing arguments, rely on the existing theory for this dynamics. Let us give some more details on the strategy. Recall that Dirac bispinors in dimension $n+1$ are maps from $\mathcal M$ to $\C^M$ with $M$ an integer bigger than $2^{\lfloor \frac{n+1}{2}\rfloor}$ (in Section 2 we will review the construction of the Dirac operator on curved spaces). Due to \eqref{square}, it is often useful to exploit the identity
\begin{equation}\label{squartrick}
(i \partial_t u - \mathcal{D} u-m\beta u)(i \partial_t u + \mathcal{D} u+m\beta u)=\big(-\partial_t^2   + \lap_h  - \frac14 \mathcal R_h u-m^2\big)\mathbb{I}_M u 
\end{equation}
where $\mathbb{I}_M$ denotes the $M$-dimensional identity matrix, so that if $u$ solves equation \eqref{diraceqcurv} then $u$ also solves system  
\begin{equation}\label{quaddir}
\begin{cases}
-\partial_t^2 u  + \lap_h u - \frac14 \mathcal{R}_h u -m^2 u = 0.\\
u(0,x)=u_0(x),\\
\partial_t u(0,x)=(\mathcal{D} +m)u_0(x)
\end{cases}
\end{equation}
which shows the close relationship between the Dirac and wave/Klein-Gordon flows. This is sometimes referred to as the ``squaring trick", and turns out to be extremely useful, at least in the flat case, to reduce the study of the algebraically rich dynamics of the Dirac equation to the much easier one of the Klein-Gordon one. Let us stress the fact that in this non-flat setting the operator $\Delta_h$ is the bispinorial Laplacian, and not the scalar one; as a consequence, it is not straightforward to adapt the results known for the wave/Klein-Gordon equation on manifolds to deal with the Dirac one. Nevertheless, by using separation of variables, in some symmetric cases it is possible to bring this strategy at a ``radial" level: we intend to walk this path here. However, this plan is not going to work in the ``general" setting of assumptions {\bf (A1)} (the assumptions taken on the metric in \cite{cacdescurv1}, see \eqref{assh} below), and it will force us to impose stronger ones.

\medskip

Before stating our main results, let us recall some basic (and classical) facts about the decomposition of the Dirac operator. On 3-dimensional spherically symmetric manifolds, i.e. if the metric enjoys the structure \eqref{sphermet}, the Dirac equation can be written in the convenient form
$$
i\partial_t \psi  = H_\varphi \psi
$$
where 
$$
H_\varphi =\begin{pmatrix} m & -i \sigma_3 \Big( \partial_r + \frac{\varphi'}{\varphi} \Big) + \frac1{\varphi}\mathcal{D}_{\mathbb{S}^{2}}\\
i \sigma_3 \Big( \partial_r + \frac{\varphi'}{\varphi} \Big) + \frac1{\varphi}\mathcal{D}_{\mathbb{S}^{2}}& -m \end{pmatrix}.
$$ 
Here $\sigma_3$ is one of the Pauli matrices:
 $$\sigma_3 =\begin{pmatrix} 1 &0\\
0& -1 \end{pmatrix}$$
and $\mathcal{D}_{\mathbb{S}^2}$ is the Dirac operator on the sphere $\mathbb{S}^{2}$ (see \cite{thaller} Section 4.6 and \cite{daude1}). It is well-known that the operator $\mathcal{D}_{\mathbb{S}^{2}}$ can be diagonalized (see \cite{camporesi}) : as a consequence, one has the following natural decomposition
\begin{equation}\label{spherdec}
L^2(\mathbb{R}^3)^4\cong\bigoplus_{\mu,j_\mu} L^2((0,+\infty),\varphi^2(r)dr)\otimes\mathcal{H}_{\mu,j_\mu}
\end{equation}
where the indexes are $\mu\in\Z^*=\Z\backslash\{0\}$, $j_\mu\in\{-|\mu|+1,-|\mu|,\dots,|\mu|\}$, and the two-dimensional Hilbert spaces $\mathcal{H}_{\mu,j_\mu}$ are generated by two orthogonal functions $\{\Phi^+_{\mu,j_\mu},\Phi^-_{\mu,j_\mu}\}$ that essentially are normalized spherical harmonics.
The action of $H_\varphi$ on the spaces $H^1(\varphi(r)^2dr) \otimes \text{Vect} (\Phi^+_{\mu,j_\mu},\Phi^-_{\mu,j_\mu})$ is given by
\begin{equation}\label{raddir2}
h_{\mu} = \begin{pmatrix} m & -\Big(\partial_r + \frac{\varphi'}{\varphi} \Big) +\frac{\mu}{\varphi} \\ \Big(\partial_r + \frac{\varphi'}{\varphi} \Big) +\frac{\mu}{\varphi} &  -m \end{pmatrix}
\end{equation}
where the $\mu\in\Z^*$ are the eigenvalues of the angular operator $\mathcal{D}_{\mathbb{S}^2}$ (notice that we are using a slightly different but equivalent decomposition with respect to \cite{thaller} and \cite{cacdescurv1}, that allows a much easier generalization).  More in general, this decomposition holds in the setting of warped product metrics \eqref{warped} in dimension $n\geq3$. Indeed, there exists a decomposition of $L^2(\K^{n-1})$, 

\[
L^2(\K^{n-1}) = \bigoplus_{\mu, j_\mu} \mathcal H_{\mu,j_\mu}
\]
where $\mu$ is taken over the spectrum of $\mathcal D_{\mathbb{K}^{n-1}}$ (which is purely discrete) and where $j_\mu \in [1,r_\mu] \cap \N$ where $r_\mu$ is the multiplicity of $\mu$. On $\mathcal H_{\mu,j_\mu}$, the action of $\mathcal D_\Sigma$ can be represented by $h_\mu$. Subsection \ref{subsec:sepvar} will be devoted to present an overview of the topic.

\subsection{Main results}

We are now ready to state the main results. For a definition of functional spaces, we refer to Subsection \ref{subsec:notations}.

\begin{definition}\label{def:admP} We say that the triple $(p,q,m)$ is admissible, either if $m=0$
\[
\frac2{p} + \frac{n-1}{q} = \frac{n-1}2, \quad  p\geq 2, \quad q\in [2,\infty) , 
\]
or if $m\neq 0$
\[
\frac2{p} + \frac{n}{q} = \frac{n}2, \quad p,q\geq2.
\]
\end{definition}

\begin{remark} Note that the results are also true when $m\neq 0$ in the first case. \end{remark}

The first result we prove is a global-in-time Strichartz estimate for the Dirac flow restricted to eigenspaces of the operator $\mathcal{D}_{\mathbb{K}^{n-1}}$.

\begin{theorem}\label{teo1}  Let $(\mathcal{M},g)$ be a Lorentzian manifold of dimension $n+1\geq 4$ defined by $\mathcal{M}=\mathbb{R}_t\times \Sigma$, with $(\Sigma,h)$ a warped product, that is a Riemannian manifold in the form $\Sigma=\mathbb{R}_r^+\times\mathbb{K}^{n-1}$   where $\mathbb{K}^{n-1}$ is an $(n-1)$-dimensional compact spin manifold, and $\Sigma$ is equipped with the Riemannian metric
\begin{equation}\label{metric}
dr^2+\varphi(r)^2d\omega_{\mathbb{K}^{n-1}}^2
\end{equation}
with $\varphi:\mathbb{R}^+\rightarrow \mathbb{R}^+$ and $d\phi^2$ the Riemannian metric on $\mathbb{K}^{n-1}$. We assume that $\varphi$ is $\mathcal C^2$, that $\varphi(0) = 0$, and $\varphi'(0) = 1$, we also assume that $\frac{\varphi'}{\varphi}$ is bounded outside a neighborhood of $0$.\vspace{0.1cm}

Let $\mu$ be in the spectrum of $\mathcal{D}_{\mathbb{K}^{n-1}}$ and assume that the operator $\mathcal{D}_{\K^{n-1}}$ has no eigenvalue $\mu$ with $|\mu|<\frac12$ (see   \cite[Theorem 3.2]{chou}). 
 Let $V_{\mu} = \frac{\mu(\mu+ \varphi')}{\varphi^2}$ and 
\begin{equation}\label{crucialass}
\delta_\varphi(\mu) = \min(1,\inf (4r^2{V}_{\mu} + 1),\inf (-4r^2 {V}_{\mu}- 4r^3{V}_{\mu}'+1)).
\end{equation}
Assume that the function $\varphi$ in the metric \eqref{metric} satisfies
\begin{equation}\label{crucialcond}
\displaystyle\lim_{r\rightarrow+\infty} V_{\mu}=0,\quad\delta_\varphi(-\mu),\delta_\varphi(\mu) > 0,\quad 4r^2{V}_{\mu} \in L^\infty.
\end{equation}

 Then, for any admissible triple $(p,q,m)$ in Definition \ref{def:admP}, there exists a constant $C$ depending only on $m,p,q, \varphi$ (but not on $\mu$) such that for all $v_0 \in H^{1/2}_{\varphi}$,
\begin{equation}\label{strichartz1}
\begin{split}
\big\| \Big( \frac{\varphi(r)}{r}\Big)^{\frac{(n-1)}2\left(1-\frac2q\right)} &e^{-ith_{\mu}}v_0 \big\|_{L^p(\R,W^{1/q-1/p,q}_\varphi )} \\
&\leq C |\mu|^{5/p+\varepsilon}(\delta_\varphi(\mu)^{1/p+\varepsilon}+\delta_\varphi(-\mu)^{1/p+\varepsilon}) \|v_0\|_{H^{1/2}_\varphi},
\end{split}
\end{equation} 
with $\varepsilon>0$ if $m=0$ in $3d$, otherwise $\varepsilon=0$, and $W^{1/q-1/p,q}_\varphi $ and $H^{1/2}_\varphi$ are Sobolev spaces on the manifold $\Sigma$ for radial functions defined in Subsection \ref{subsec:notations}.

\end{theorem}

\begin{remark}
The condition on $\varepsilon$ in estimate \eqref{strichartz1} is connected to the admissibility of the endpoint triple $(p,q,m)$, as we will briefly discuss in the proof of Corollary \ref{strichartzcor}.
\end{remark}

\begin{remark}
Note that assuming that the compact manifold $\K^{n-1}$ satisfies that the (discrete) spectrum of the Dirac operator on the compact manifold $\K^{n-1}$ is included in $(-\infty,- \frac12]\cup [\frac12, \infty)$,  we get that the Dirac operator on $\Sigma$ where $\varphi=r$, is self-adjoint (see \cite{chou} and the references therein). The operator $h_\mu$ being isomorphic to an $L^\infty$ perturbation of 
\[
\tilde h_\mu = \begin{pmatrix} m & -\Big( \partial_r + \frac1{r}\Big) + \frac\mu{r} \\
\Big( \partial_r + \frac1{r}\Big) + \frac\mu{r} & -m \end{pmatrix},
\]
we get that $h_\mu$ is selfadjoint. 

This will be further commented upon in Remark \ref{rem:selfadjointness}.

\end{remark}

\begin{remark}
 When $\K^{n-1}$ is the $n-1$-dimensional sphere, then the manifold $\Sigma$ is smooth and in fact geometrically complete, which ensures the self-adjointness of the Dirac operator. Also, in this case, by relying on the endpoint Strichartz estimate proved in \cite{mach}, it is possible (and immediate) to recover the endpoint as well, namely the estimate
\begin{equation}\label{strichartznew3}
\left\|u\left(\frac{\varphi(r)}r\right)^{\frac{(n-1)}2}\right\|_{L^2_t(I, L^\infty(\R_+,L^p(\mathbb S^{n-1})))}\leq  \sqrt p |\mu|^{5/2+\varepsilon}(\delta_\varphi(\mu)^{1/2+\varepsilon}+\delta_\varphi(-\mu)^{1/2+\varepsilon}) \|u_0\|_{H^1(\Sigma)}.
\end{equation}
We omit the details.
\end{remark}

\begin{remark}
The dependence on the angular parameter $\mu$  in our Strichartz estimates (that can be ultimately intended as a loss of angular derivatives, and is most likely not sharp), is due to the method of our proof: the action of the ``radial Dirac operator" \eqref{raddir2} depends on the ``angular" eigenvalue $\mu$, and 
as a consequence the Strichartz estimates for the flow $e^{it h_\mu}$ will necessarily depend on $\mu$. The additional $\varepsilon$-loss in the massless case is due to the lack of the endpoint Strichartz estimates in this case, as indeed these estimates will be obtained by interpolation. We refer to \cite{cacdescurv1}, Section 5 for all the details.
\end{remark}

\begin{remark}
The hypothesis $\displaystyle\lim_{r\rightarrow+\infty} V_{\mu}=0$ could be removed by considering a new potential $W_{\mu}=V_{\mu}-v_{\mu}^\infty$ where $v_{\mu}^\infty=\displaystyle\lim_{r\rightarrow +\infty} V_{\mu}$, but as a matter of fact this will be implied by our forthcoming assumptions {\bf (A2)}, and therefore we do not strive to optimize on this condition, as done in \cite{danzha}.
\end{remark}

\begin{remark}
With slight additional care, the result above could be generalized in order to include spaces with conical singularities; the study of the Dirac operator in this context, mostly from the spectral point of view, has been developed in details in  \cite{chou}. The analyisis of dispersive flows on conical spaces (and on spaces with conical singularities) has seen increasing interest in recent years; we don't intend to provide a precise picture of the literature here. 
We mention that the present work has in fact originally motivated the paper \cite{kg}, in which we have analyzed the dispersive dyamics of the Klein-Gordon equation on spaces with conical singularities.  Anyway, we need to stress once more the fact that it is not possible to directly adapt those results to the context of the Dirac flow, as the Laplacian operators are in fact of a different nature (spinorial vs scalar).  
\end{remark}

The fact that the constant on the right hand side of estimate \eqref{strichartz1} is a function of $\mu$ suggests that it might be possible to prove Strichartz estimates with loss of angular derivatives: this kind of estimates are quite classical in the context of dispersive PDEs, and the local-in-time case (in dimension $3$) has been already discussed in the predecessor of this paper, that is \cite{cacdescurv1}. For the next Theorem we shall indeed restrict to the case $\mathbb{K}^{n-1}=\mathbb{S}^{n-1}$, in order to be able to resort to the well-established Littlewood-Paley theory on the sphere. It is in fact possible to ``sum" the Strichartz estimates \eqref{strichartz1} in order to obtain Strichartz estimates for general initial data by requiring additional regularity in the angular variable (we postpone to subsection \ref{subsec:notations} the precise definitions of the spaces $H^{a,b}(\Sigma)$). 

The result is the following:

\begin{theorem}\label{teo2} 
Let $(\mathcal{M},g)$ be defined by $\mathcal{M}=\mathbb{R}_t\times \Sigma$, with $(\Sigma,h)$ a spherically symmetric manifold of dimension $n\geq 3$ with metric given by \eqref{warped} and let $\varphi$ satisfy assumptions \textbf{(A2)} below.

 Let $p,q\in [2,\infty]$ and $a,b \geq 0$ such that $(p,q,m)$ is admissible. Assume either when $m=0$, $n=3$, that $\frac5{pb} + \frac{1}{2a} < 1$ or when $m\neq 0$ or $n>3$, that $\frac5{pb} + \frac{1}{2a} \leq1$. Then the solutions $u$ to \eqref{diraceqcurv} with initial data $ u_0\in H^{a,b}(\Sigma)$ satisfy the estimates
\begin{equation}\label{strichartz2}
\left\|\left(\frac{\varphi(r)}r\right)^{\frac{(n-1)}2\left(1-\frac2q\right)}u \right\|_{L^p_t(\R,W^{1/q-1/p,q}(\Sigma))}\leq C\| u_0\|_{H^{a,b}(\Sigma)}.
\end{equation} 
\end{theorem}

\begin{remark}
The analogous of Theorem \ref{teo2} could be proved in the more general case of warped products under assumption \eqref{crucialcond} provided one has a suitable Littlewood-Paley theory on the manifold $\mathbb{K}^{n-1}$. This might be the object of forthcoming works.
\end{remark}

\subsection{Admissible manifolds: discussion}\label{assub}

The assumptions on the function $\varphi$ required in \cite{cacdescurv1} to obtain local-in-time Strichartz estimates were the following:
\medskip

\textbf{Assumptions (A1).} Let $\varphi\in C^\infty(\mathbb{R^+})$ be strictly positive on $(0,+\infty)$, such that
\begin{equation}\label{assh}
\varphi(0)=\varphi^{(2n)}(0)=0,\qquad \varphi'(0)=1,\qquad \frac{\varphi'(r)}{\varphi(r)}\in L^\infty , \qquad \inf_{r\geq 1} \varphi(r) > 0.
\end{equation}

As mentioned, we won't be able to obtain global-in-time Strichartz estimates in this context, as {\bf (A1)} is not enough to ensure the necessary condition \eqref{crucialcond}. Instead, the main example we here have in mind is given by the \emph{asymptotically flat manifolds}: we thus set  the following

\medskip
\textbf{Assumptions (A2).} Let $\varphi \in C^\infty(\mathbb{R^+})$ be such that $\varphi(0) = 0$, $\varphi'(0) = 1$, and for all $k\in \N$, $\varphi^{(2k)}(0) = 0$. We assume that there exists $\varphi_1 \in \mathcal C^\infty (\R^+)$ such that
\[
\varphi : r\mapsto r(1 + \varphi_1(r))
\]
with the following assumptions on $\varphi_1$:

 \begin{itemize}
\item $\varphi_1$ is non-negative 
\medskip
\item  $\sup_{r\geq 0}(|\varphi_1(r)|+|r\varphi_1(r)'|+ |r^2\varphi_1''(r)|)\leq C$, where $C\ll 1$.

\end{itemize}

\begin{remark} The map
\[
\varphi_1 = \varepsilon\frac{r^\alpha}{\an{r}^\beta}
\]
with $\beta \geq \alpha> 0$, with $\alpha,\beta \in \N$ satisfies these assumptions.
\end{remark}

In Subsection \ref{admissasflat} we will prove the following 
\begin{proposition}\label{verification} 
Let $(\M,g)$ be defined by $\M=\mathbb{R}_t\times \Sigma$, with $(\Sigma,h)$ a warped product with metric given by \eqref{warped}. Let $\mu_0$ be the infimum of the positive part of the spectrum of the Dirac operator on $\K^{n-1}$, and assume that $\mu_0>1/2$. If $\varphi$ satisfies assumptions \textbf{(A2)} where the required smallness of  $C$ is determined by $\mu_0$, then the assumptions \eqref{crucialcond} are fullfilled.
\end{proposition}

\begin{remark}
We will provide more precise assumptions on $\varphi_1$ and in particular on the size of the constant $C$, with explicit dependence on the space dimension, at the beginning of Section \ref{sec:strich}.
\end{remark}

\begin{remark}
It is a natural question to ask whether there exist other possible choices of the function $\varphi$ that satisfy condition \eqref{crucialcond}. We will devote the appendix to a small discussion.
\end{remark}

The plan of the paper is the following: In Section \ref{sepvar} we review the separation of variables procedure for the Dirac equation in the warped products setting, and show how to reduce to the Klein-Gordon dynamics. In Section \ref{secstr} we discuss the classical Kato's argument to obtain the Strichartz estimates for the Klein-Gordon dynamics with potentials of critical decay. Finally, in Section \ref{sec:strich} we show that  asymptotically flat manifolds are admissible, and we prove Strichartz estimates for general initial data in the spherically symmetric setting.

\subsection{Notations}\label{subsec:notations}

We will use the standard notation $L^p$, $\dot{H}^s$, $H^s$, $W^{p,q}$ to denote, respectively, the Lebesgue and the homogeneous/non homogeneous Sobolev spaces of functions from $\R^n$ to $\C^M$. We will use the same notation to denote these functional spaces on the (spatial) manifold $(\Sigma,h)$, which is in our structure \eqref{struct}, i.e. with time and space already decoupled, by adding the dependence $L^p(\Sigma)$, $\dot{H}^s(\Sigma)$, $H^s(\Sigma)$, $W^{p,q}(\Sigma)$: e.g., the norm $L^p(\Sigma)$ will be given by
$$
\|f\|_{L^p(\Sigma)}^p:=\int |f(x)|^p \sqrt{\det(h(x))}dx
$$
and so on. 

The space $L^2(\Sigma)$ is thus endowed with the usual Hilbertian structure.

The space $\dot H^1(\Sigma)$,  is induced by the norm
\[
\|f\|_{\dot H^1(\Sigma)}^2 : =  \|\sqrt{h^{ij}\an{D_i f,D_j f}_{\C^M}}\|_{L^2(\Sigma)}
\]
where the $D_j$ are covariant derivatives for Dirac bispinors.

The space $W^{1,p}$, $p\in [1,\infty]$ is induced by
\[
\|f\|_{W^{1,p}(\Sigma)} =  \|\sqrt{h^{ij}\an{D_i f,D_j f}_{\C^M}}\|_{L^p(\Sigma)} + \|f\|_{L^p(\Sigma)}.
\]
The spaces $\dot H^s(\Sigma)$ and $W^{s,p}(\Sigma)$ with $s\in [-1,1]$ are defined by interpolation and duality.

Due to the warped product structure of the metric \eqref{warped}, for a radial function $f_{rad}(|x|)$ we define
$$
\|f_{rad}\|_{L^p_\varphi}^p:=\int_0^{+\infty} |f_{rad}(r)|^p \varphi(r)^{n-1}dr \sim \|f_{rad}\|_{L^p(\Sigma)}^p .
$$
For the Sobolev spaces, we use the compatible notations
\[
\|f_{rad}\|_{\dot H^1_\varphi} := \|\partial_r f_{rad}\|_{L^2(\Sigma)}  
\]
and 
\[
\|f_{rad}\|_{W^{1,p}_\varphi} := \|\partial_r f_{rad}\|_{L^p(\Sigma)} + \|f_{rad}\|_{L^p(\Sigma)}.
\]
We define $\dot H^s_\varphi$ and $W^{s,p}_\varphi$, $s\in (0,1)$, by interpolation, and $\dot H^s_\varphi, W^{s,p}_\varphi$, $s\in [-1,1], p\in (1,\infty)$, by duality. 

It is important to note that because of the behavior of $\varphi$ at $0$, we have 
\[
\|f_{rad}\|_{W^{s,p}_\varphi} = \left\|\Big(1-\frac1{\varphi^{n-1}}\partial_r(\varphi^{n-1} \partial_r) \Big)^{s/2}f\right\|_{L^p(\Sigma)}.
\]
When $\K^{n-1}$ is the sphere, due to the absence of singularity at $0$, this extends to the whole manifold, as in
\[
\|f\|_{W^{s,p}(\Sigma)} = \left\|\Big(1-\frac1{\varphi^{n-1}}\partial_r(\varphi^{n-1} \partial_r) - \frac1{\varphi^2} \Delta_{\mathbb S^2}\Big)^{s/2}f\right\|_{L^p(\Sigma)}
\]
where $\Delta_{\mathbb S^2}$ is the Laplace-Beltrami operator on the sphere.

Note that since we are dealing with vectors in $\C^M$, $|f(x)|$ should be understood as 
$$
|f(x)| = \sqrt{\an{f(x),f(x)}_{\C^M}}.
$$

The norms in time will be denoted by $L^p_t$. The mixed Strichartz spaces will be standardly denoted by $ L^p_tL^q(\Sigma)=L^p(I;L^q(\Sigma,\C^M))$.

We finally introduce the spaces $H^{a,b}$ for $a \in[-1,1], b\in \R$ by defining the norms
$$
\|f\|_{H^{a,b}(\Sigma)} = \Big( \|f\|_{H^a(\Sigma)}^2 + \|(-\lap_{\mathbb S^{n-1}})^{b/2}f\|_{L^2(\Sigma)} \Big)^{1/2}.
$$

\section{The setup: separation of variables \\ and reduction to Klein-Gordon}\label{sepvar}

The construction of the Dirac operator on a 4D manifold is a rather delicate task in general, and requires the introduction of the so called \emph{vierbein} which, essentially, define some proper frames that connect the metric of the manifold $(\mathcal{M},g)$ to the Lorentzian one $\eta$; details can be found in the predecessor of this paper, \cite{cacdescurv1}, and in \cite{parktoms}. In order to properly define those frames, also known as Cartan's formalism, one needs the hypothesis that the manifold admits a spin structure: we will take this as an assumption. The fact that admitting a spin structure is a homological property has been proved and commented upon in \cite{spinstructure}. In fact, $\Sigma$ (and $\mathcal{M}$) inherit a spin structure from the spin structure of $\K^{n-1}$; we will explain this in the next subsection.

In this section we show how to use separation of variables and the classical spectral theory for the Dirac equation on compact manifolds to reduce the study of the dynamics of Dirac equation on warped products to the one of a system of radial Klein-Gordon equations. We refer the reader to \cite{chou,bar,thaller} for further details on various aspects we will discuss.

\subsection{The separation of variables}\label{subsec:sepvar}

We start by recalling that the Dirac operator in a Lorentzian manifold $(\mathcal{M},g)$ of dimension $n+1$ admitting a spin structure and with decoupled space and time writes 
$$
\mathcal{D}  = m\gamma^0  - i\gamma^0\underline \gamma^j D_j  .
$$
where the implicit summation on $j$ is taken from $1$ to $n$. Above, $m\in \R$ is the mass of the electron, $\gamma^0$ is a self-adjoint matrix of size $M\times M$  with $M=2^{\lfloor (n+1)/2 \rfloor}$ with values in $\C$ whose square is the identity, $\underline \gamma^j$ are anti-hermitian matrix bundles that satisfy
$$
\forall j,k =1,\hdots ,n \quad \{\underline \gamma^j,\underline{\gamma^k}\} = 2 g^{jk},\quad \forall j=1,\hdots, n \quad \{\gamma^0,\underline \gamma^j\} = 0
$$
and $D_j$ are covariant derivatives for spinor bundles.

Writing $\mathcal M, g$ as $\mathcal M = \R_t \times \Sigma$ and
\[
g = \begin{pmatrix}
1 & (0) \\ (0) & -h 
\end{pmatrix}
\]
where $h$ is the (Riemannian)  metric of $\Sigma$, we endow the spinorial Riemannian manifold $\Sigma, h$ with a vierbein $e^j_{\; a}$ (chosen such that for all $j,k$, $e^j_{\; a}\delta^{ab} e^k_{\; b} = h^{jk}$), we fix
$$
\underline \gamma^j = e^j_{\; a} \gamma^a.
$$
The implicit summation for $a$ is taken from $1$ to $n$. The family $(\gamma^a)_{0\leq a\leq n}$ satisfies the anticommutation relations: 
$$
\{\gamma^a,\gamma^b\} = 2\eta^{ab}
$$
where 
$$
\eta = \begin{pmatrix}
1 & &  & \\  & -1 & (0) & \\ & (0) & \ddots & \\  & & & -1
\end{pmatrix}
$$
is the Minkowski metric in $\R^{1+n}$. Writing $\alpha^0 = \gamma^0$ and $\alpha^a = \gamma^0 \gamma^a$, we have that the family $(\alpha^a)_{0\leq a \leq n}$ satisfy the \emph{canonical anticommutation relations} 
$$
\{\alpha^a,\alpha^b\} = 2\delta^{ab}
$$
and are self-adjoint matrices. What is more, the Dirac operator now writes
$$
\mathcal{D} = m\alpha^0 -i e^j_{\; a} \alpha^a D_j.
$$

The minimal dimension for such a family of matrices is $2^{\lfloor(n+1)/2 \rfloor}$. An easy way to see that this dimension is big enough is to consider for $n=2$, the Pauli matrices 
$$
\alpha^0 = \sigma_1 = \begin{pmatrix} 1 &0 \\ 0 & -1\end{pmatrix},\quad \alpha^1 = \sigma_2 = \begin{pmatrix} 0 & i \\-i & 0 \end{pmatrix} , \quad \alpha^2 = \sigma_3 = \begin{pmatrix}
0 & 1 \\ 1 & 0 
\end{pmatrix} 
$$
and for $n=2k+2$ even, given a family  $(\tilde \alpha^a)_{0\leq a \leq 2k}$ of self-adjoint matrices of size $K\times K$ satisfying canonical anticommutation relations, the matrices written by block
\begin{multline}\label{eventoeven}
\alpha^0 = \begin{pmatrix}
\textrm{Id}_K & (0) \\ (0) & -\textrm{Id}_K
\end{pmatrix}, \\
\forall a=0,\hdots ,2k , \alpha^{a+1} = \begin{pmatrix}
(0) & \tilde \alpha^a \\ \tilde \alpha^a & (0)
\end{pmatrix}, \quad
\alpha^n = \begin{pmatrix}
(0) & i \textrm{Id}_K \\ -i \textrm{Id}_K & (0)
\end{pmatrix}.
\end{multline}

Therefore, a natural way to pass from dimension $n = 2k$ even to $n+1 = 2k+1$ odd is to pass from the family of matrices $(\tilde \alpha^a)_{0\leq a \leq n}$ to 
$$
\alpha^0 = \begin{pmatrix}
\textrm{Id}_K & (0) \\ (0) & -\textrm{Id}_K
\end{pmatrix}, \quad
\forall a=0,\hdots ,2k , \alpha^{a+1} = \begin{pmatrix}
(0) & \tilde \alpha^a \\ \tilde \alpha^a & (0)
\end{pmatrix};
$$
and to pass from dimension $n+1 = 2k+1$ odd to dimension $n+2$ even is simply to add the matrix 
$$
\alpha^n = \begin{pmatrix}
(0) & i \textrm{Id}_K \\ -i \textrm{Id}_K & (0)
\end{pmatrix}.
$$
However, it is also natural to pass from an odd to an even dimension in the same way as to pass from an even to an odd. The reason is that, because of the theory of Clifford algebras, the algebra generated by the family $(\alpha^a)_{0\leq a\leq n+2}$ defined as in \eqref{eventoeven} is canonically isomorphic to the one generated by
$$
\begin{pmatrix}
\textrm{Id}_{2K} & (0) \\ (0) & -\textrm{Id}_{2K}
\end{pmatrix}, \quad
\forall a=0,\hdots ,n+1 , \begin{pmatrix}
(0) & \alpha^a \\ \alpha^a & (0)
\end{pmatrix}.
$$
We now consider the following setting: $(\Sigma,\sigma)$ is a warped product, that is a Riemannian manifold in the form $\Sigma=\mathbb{R}^+_r\times\mathbb{K}^{n-1}_\phi$ where $\mathbb{K}^{n-1}$ is a $(n-1)$-dimensional compact spin manifold, and $\Sigma$ is equipped with the Riemannian metric
$$
dh^2=dr^2+\varphi(r)^2d\phi^2
$$
where $\varphi :\mathbb{R}^+\rightarrow \mathbb{R}^+$ and $d\phi^2$ is the Riemannian metric over $\mathbb{K}^{n-1}$. In other words,
\[
h = \begin{pmatrix}
1 & (0) \\ (0) & \varphi^2 \kappa
\end{pmatrix}
\]
where $\kappa$ is the Riemannian metric of $\K^{n-1}$.

In the case that interests us, we assume that a vierbein $\tilde e = (\tilde e_j^{\;a})$ has been set for $\mathbb{K}^{n-1}$, that we assume admits a spin structure. As the equation is covariant, we may choose any convenient vierbein for $\Sigma$: we use as a vierbein for $\Sigma$
$$ e=
\begin{pmatrix}
1 & (0) \\ (0) & \varphi(r)\tilde e
\end{pmatrix}.
$$
We set $(\tilde\alpha^a)_{0\leq a\leq n-1}$ a family of matrices satisfying canonical anticommutation relations and 
$$
\alpha^0 = \begin{pmatrix} \textrm{Id} & (0) \\ (0) & -\textrm{Id}\end{pmatrix},\quad \forall a=0,\hdots ,n-1,\; \alpha^{a+1} = \begin{pmatrix} (0) &  \tilde \alpha^a \\ \tilde \alpha^a & (0) \end{pmatrix}.
$$
We set also 
$$
\gamma^0  = \alpha^0 , \forall a=1,\hdots ,n,\; \gamma^a  = \alpha^0 \alpha^a ,
$$
and finally 
$$
\tilde \gamma^0 = \tilde \alpha^0, \quad \forall a=1,\hdots, n ,\; \tilde \gamma^a = \tilde \alpha^0 \tilde \alpha^a.
$$
We recall that the covariant derivatives for Dirac spinors are given by
$$
D_\mu = \partial_\mu +i \omega_\mu^{\; ab} \Sigma_{a,b}.
$$
where $\omega$ is the spin connection and 
$$
\Sigma_{a,b} = -\frac{i}{8} [\gamma_a,\gamma_b].
$$

We have for all $a,b =1,\hdots, n$
\begin{multline*}
[\gamma^a,\gamma^b] = [\alpha^0 \alpha^a, \alpha^0\alpha^b] = [\alpha^b,\alpha^a] = \begin{pmatrix}
[\tilde \alpha^{b-1},\tilde \alpha^{a-1}] & (0) \\ (0) &  [\tilde \alpha^{b-1},\tilde \alpha^{a-1}]
\end{pmatrix} = \\
\begin{pmatrix}
[\tilde \gamma^{a-1}, \tilde \gamma^{b-1}] & (0) \\ (0) & [\tilde \gamma^{a-1}, \tilde \gamma^{b-1}]
\end{pmatrix}.
\end{multline*}
Therefore, we have 
$$
\Sigma_{a,b}  = \begin{pmatrix}
\tilde \Sigma_{a-1,b-1} & (0) \\ (0) & \tilde \Sigma_{a-1,b-1}
\end{pmatrix}
$$
where $\tilde \Sigma_{a,b} = -\frac{i}{8} [\tilde \gamma^a, \tilde \gamma^b ]$.

We also have 
$$
de^a + \omega^a_{\; b} \wedge e^b = 0.
$$
Since $e^1  = dr$, we have $de^1 = 0$ and thus 
\[
\omega^1_{\; b} \wedge e^b = 0.
\]
Therefore, we get $\omega^1_{\; b} \sim e_{\; b}$  for all $b$ and then $\omega_1^{\; 1b} = 0$ for all $b\geq 1$.

Since for all $a>1$, we have $e^a = \varphi(r)\tilde e^{a-1}$, we get
$$
de^a = \varphi' e^1 \wedge \tilde e^{a-1} + \varphi d\tilde e^{a-1} = \varphi' e^1 \wedge \tilde e^{a-1} - \varphi \tilde \omega^{a-1}_{\; b} \wedge \tilde e^b
$$
and thus
$$
\omega^a_{\; b} \wedge e^{b} = \varphi' \tilde e^{a-1} \wedge e^1  + \varphi \tilde\omega^{a-1}_b\wedge \tilde e^{b} =\varphi' \tilde e^{a-1} \wedge e^1  +  \tilde\omega^{a-1}_b\wedge  e^{b+1} .
$$
Therefore,
$$
\omega^a_{\; 1} = \varphi' \tilde e^{a-1}\sim e^a \Rightarrow \omega_1^{\; a1} = 0 \textrm{ and for all }j>1, \; \omega_j^{\; a1} =\varphi' \tilde e_{j-1}^{a-1} 
$$
and for all $b>1$,
$$
\omega^a_{\; b} = \tilde \omega^{a-1}_{\; (b-1)} \Rightarrow \omega_1^{\; ab} = 0 \textrm{ and for all }j>1 ,\; \omega_j^{\; ab} = \tilde \omega_{j-1}^{(a-1)(b-1)}.
$$
Summing up, we get $D_1 = \partial_r$ and for all $j>1$,
$$
D_j = \partial_j +2 i\varphi' \tilde e_{j-1}^{\; a} \Sigma_{1(a+1)} + i   \tilde \omega_{j-1}^{(a-1)(b-1)} \Sigma_{a,b}.
$$
Since 
$$
\Sigma_{1,(a+1)} = \begin{pmatrix}
\tilde \Sigma_{0,a} & (0) \\ (0) & \tilde \Sigma_{0,a}
\end{pmatrix} \textrm{ and } \Sigma_{a,b}  = \begin{pmatrix}
\tilde \Sigma_{a-1,b-1} & (0) \\ (0) & \tilde \Sigma_{a-1,b-1}
\end{pmatrix},
$$
we have
$$
D_j = \begin{pmatrix}
\tilde D_{j-1} + 2i\varphi'\tilde e_{j-1}^{\; a} \tilde \Sigma_{0,a}& (0) \\ (0) & \tilde D_{j-1} +2i\varphi'\tilde e_{j-1}^{\; a} \tilde \Sigma_{0,a}
\end{pmatrix}
$$
where $\tilde D_j$ are covariant derivatives for spinor bundles over $\K^2$.

We deduce
\begin{multline*}
ie^j_{\; b} \alpha^b D_j = i\alpha_1\partial_r + i\frac1{\varphi} \tilde e^{j-1}_{\; b} \begin{pmatrix}
(0) & \tilde \alpha^b \\ \tilde \alpha^b & (0) 
\end{pmatrix} \begin{pmatrix}
\tilde D_{j-1} + 2i\varphi'\tilde e_{j-1}^{\; a}\tilde \Sigma_{0,a}& (0) \\ (0) & \tilde D_{j-1} +2i\varphi'\tilde e_{j-1}^{\; a} \tilde\Sigma_{0,a}
\end{pmatrix} \\ =
i\alpha_1\partial_r +\frac1{r} \begin{pmatrix}
(0) & i \tilde e^j_{\; b}\tilde \alpha^b (\tilde D_j + i\tilde e_j^{\; a}\tilde \Sigma_{0,a}) \\ i \tilde e^j_{\; b}\tilde \alpha^b (\tilde D_j + i\tilde e_j^{\; a}\tilde \Sigma_{0,a}) &(0)
\end{pmatrix}.
\end{multline*}
Using that $\tilde e^j_{\; b}\tilde e_j^{\; a} = \delta^{a}_{\; b}$ and that $\tilde \alpha^a \tilde \Sigma_{0,a} =  -i \frac{n-1}4 \tilde \alpha^0 $, we deduce 
$$
\mathcal{D}  = m \alpha^0 + i\alpha_1\partial_r + \begin{pmatrix}
(0) & \frac1{\varphi(r)}\mathcal{D}_{\K^{n-1}} + i\frac{n-1}{2}\frac{\varphi'(r)}{\varphi(r)} \tilde\alpha^0 \\  \frac1{\varphi(r)}\mathcal{D}_{\K^{n-1}} + i\frac{n-1}{2}\frac{\varphi'(r)}{\varphi(r)} \tilde\alpha^0&(0)
\end{pmatrix}
$$
where $\mathcal{D}_{\K^{n-1}}$ is the Dirac operator on $\mathbb{K}^{n-1}$.  We thus get
\begin{equation}\label{ridir}
\mathcal{D} = \begin{pmatrix}
m &  i\tilde\alpha^0 \Big( \partial_r + \frac{n-1}{2}\frac{\varphi'(r)}{\varphi(r)} \Big) + \frac1{\varphi(r)}\mathcal{D}_{\K^{n-1}} \\  i\tilde\alpha^0 \Big( \partial_r + \frac{n-1}{2}\frac{\varphi'(r)}{\varphi(r)} \Big) + \frac1{\varphi(r)}\mathcal{D}_{\K^{n-1}}  & -m
\end{pmatrix}.
\end{equation}
Now the key step (for us) consists in ensuring that the operator $\mathcal{D}_{\K^{n-1}}$ can in fact be diagonalized. In the case $\K^{n-1}$ being the $2$-dimensional unit sphere, this fact is classical and well-known, the eigenvalues and eigenfunctions are explicit (see e.g. \cite{thaller} or \cite{camporesi}). In the general case, we can nevertheless evoke the following result,  that can be found e.g in \cite[Theorem 5.27]{roe}.

\begin{proposition}
Let $\K^{n-1}$ be a smooth compact manifold and let $H$ be a Dirac operator on it. Then there is a direct sum decomposition of $H$ into a sum of countably many orthogonal spaces $H_\mu$, each of which is a finite-dimensional space of smooth sections and is an eigenspace for $H$ with eigenvalue $\mu$. The eigenvalues $\mu$ form a discrete subset of $\mathbb{R}$.
\end{proposition}

Let $\mu > 0$ be in the spectrum of $\mathcal{D}_{\K^{n-1}}$, we fix $(\psi_{\mu,j})_{j}$ an orthogonal basis of the eigenspace of $\mathcal{D}_{\K^{n-1}}$ with eigenvalue $\mu$. We set $\psi_{-\mu,j} = i\tilde \alpha^0 \psi_{\mu,j}$. Since $\tilde \alpha^0$ anticommutes with $\mathcal{D}_{\K^{n-1}}$, we get that $\psi_{-\mu,j}$ is an eigenfunction of $\mathcal{D}_{\K^{n-1}}$ with eigenvalue $-\mu$. Note that for all $\mu $ in the spectrum of $\mathcal{D}_{\K^{n-1}}$, since $(i\alpha^0)^2 = -1$, we have for $\mu > 0$, $(1+i\tilde \alpha^0) \psi_{\mu,j} = \psi_{\mu,j} + \psi_{-\mu,j} $ and $(1-i\tilde \alpha^0)\psi_{-\mu,j} = \psi_{-\mu,j} - \psi_{\mu,j}$. Similarly,
\[
(1-i\tilde \alpha^0) \psi_{\mu,j} = \psi_{\mu,j} - \psi_{-\mu,j} \quad \textrm{and } \quad (1-i\tilde \alpha^0) \psi_{-\mu,j} = \psi_{-\mu,j} + \psi_{\mu,j}.
\]
 Therefore, the family
$$
\mathcal B = \left( \begin{pmatrix}\frac{1+i\tilde \alpha^0}{\sqrt 2} \psi_{\mu,j} \\0 \end{pmatrix},\begin{pmatrix} 0 \\ -\frac{1-i\tilde \alpha^0}{\sqrt 2} \psi_{\mu,j} \end{pmatrix} \right)_{\mu \in Sp(\mathcal{D}_{\mathbb{K}^{n-1}}),j}
$$
forms an orthonormal basis of $L^2 (\K^{n-1}, \C^M)$.

We deduce that we have the decomposition
\[
L^2(\Sigma, \C^M) = \bigoplus_{\mu,j} \mathcal H_{\mu,j}
\]
where $\mathcal H_{\mu,j}$ is the tensor product of $L^2_\varphi$ ( the $L^2$ maps of $\R_+$ with measure $\varphi^2 dr$) and with values in $\C$ ; and the vector space generated by 
\[
\left( \begin{pmatrix}\frac{1+i\tilde \alpha^0}{\sqrt 2} \psi_{\mu,j} \\0 \end{pmatrix},\begin{pmatrix} 0 \\ -\frac{1-i\tilde \alpha^0}{\sqrt 2} \psi_{\mu,j} \end{pmatrix} \right).
\]

In other words, any map $u\in L^2(\Sigma, \C^M)$ may be written as
\[
u(r,\omega) = \sum_{\mu,j} u_{\mu,j}^+ (r) \begin{pmatrix}\frac{1+i\tilde \alpha^0}{\sqrt 2} \psi_{\mu,j}(\omega) \\0 \end{pmatrix} + u_{\mu,j}^- (r)\begin{pmatrix} 0 \\ -\frac{1-i\tilde \alpha^0}{\sqrt 2} \psi_{\mu,j}(\omega) \end{pmatrix}
\]
where $\omega\in \K^{n-1}$, and $u_{\mu,j}^\pm \in L^2_{\varphi}$ are such that
\[
\sum_{\mu,j} \|u_{\mu,j}^+ (r)\|_{L^2_\varphi}^2 + \|u_{\mu,j}^- (r)\|_{L^2_\varphi}^2 < \infty.
\]

 For any $\mu \in Sp(\mathcal{D}_{\K^{n-1}})$ and any $j$, and $f(r)$ a radial test function, we have
\begin{multline*}
\mathcal{D} \Big[f \begin{pmatrix} (1+i\tilde \alpha^0 ) \psi_{\mu,j} \\ 0 \end{pmatrix} \Big]= \\
mf  \begin{pmatrix} (1+i\tilde \alpha^0 ) \psi_{\mu,j} \\ 0 \end{pmatrix} + \left(\Big[ -\Big(\partial_r + \frac{n-1}{2}\frac{\varphi'(r)}{\varphi(r)} \Big) + \frac{\mu}{\varphi(r)} \Big] f \right) \begin{pmatrix} 0 \\ (1-i\tilde \alpha^0 ) \psi_{\mu,j} \end{pmatrix}
\end{multline*}
and 
\begin{multline*}
\mathcal{D} \Big[f  \begin{pmatrix} 0 \\ (1-i\tilde \alpha^0 ) \psi_{\mu,j} \end{pmatrix} \Big]=\\
 -mf  \begin{pmatrix} 0 \\ (1-i\tilde \alpha^0 ) \psi_{\mu,j} \end{pmatrix} +\left( \Big[ \Big( \partial_r + \frac{n-1}{2}\frac{\varphi'(r)}{\varphi(r)} \Big) + \frac{\mu}{\varphi(r)} \Big] f \right) \begin{pmatrix} (1+i\tilde \alpha^0 ) \psi_{\mu,j}  \\ 0 \end{pmatrix}.
\end{multline*}
We are thus left with studying the dispersion of the equation
\begin{equation}\label{systf}
i\partial_t F+h_\mu F=0,
\end{equation}
with 
\begin{equation}\label{reducedgeneraldirac}
h_\mu=
\begin{pmatrix} m & -\Big(\partial_r +  \frac{n-1}{2}\frac{\varphi'(r)}{\varphi(r)} \Big) + \frac{\mu}{\varphi(r)} \\\Big(\partial_r +  \frac{n-1}{2}\frac{\varphi'(r)}{\varphi(r)} \Big) + \frac{\mu}{\varphi(r)}& -m \end{pmatrix}
\end{equation}
for any $\mu \in Sp(\mathcal{D}_{\mathbb{K}^{n-1}})$.

 \begin{remark}\label{rem:selfadjointness}
If we take $\psi_\mu$ an eigenfunction of $\mathcal{D}_{\mathbb{K}^{n-1}}$ on $\mathbb{K}^{n-1}$ with eigenvalue $\mu\neq 0$ and we suppose that $\theta=f(r) \begin{pmatrix} (1+i\tilde \alpha^0 ) \psi_\mu \\ 0 \end{pmatrix} $ is an eigenspinor of $\mathcal{D}^2$ with eigenvalue $\rho^2\neq 0$, then we have that $f$ satisfies the following ODE
$$
f''+\frac{n-1}r+\left[\rho^2-\left(\mu^2-\mu-\frac{n^2-4n+3}4\right)\frac1{r^2}\right]f=0
$$
which has solutions $f(r)=\gamma^c J_{\pm\nu^+}(\rho r)$, where $c=(2-n)/2$, $\nu^+=|2\mu-1|/2$ and $J_{\nu^+}$ is the standard Bessel function of order $\nu^+$. 
 Analogously, assuming that now $\theta=f(r) \begin{pmatrix} 0 \\  (1-i\tilde \alpha^0 ) \psi_\mu  \end{pmatrix} $ is an eigenspinor of $\mathcal{D}^2$ with eigenvalue $\rho^2\neq 0$, we see that $f$ satisfies the ODE
 $$
f''+\frac{n-1}r+\left[\rho^2-\left(\mu^2+\mu-\frac{n^2-4n+3}4\right)\frac1{r^2}\right]f=0
$$
 which has solutions $f(r)=r^c J_{\pm\nu^+}(\rho r)$, with $c$ as before and $\nu^+=|2\mu+1|/2$. These equations recall once more the connection with the Klein-Gordon equation, which now has been brought at the ``radial" level. In particular, in \cite{kg} these equations are the starting point in order to prove the crucial local smoothing estimates for the Klein-Gordon equation; nevertheless, we stress once again the fact that the argument of deducing dispersive estimates for the Dirac flow from the corresponding Klein-Gordon ones does not work for free, as indeed the Laplace operator that comes into play when squaring the Dirac operator is the spinorial one (and not the standard scalar one that we dealt with in \cite{kg}).

We can explicitly write down the positive and negative eigenspinors of the operator $\mathcal{D}_{\K^{n-1}}$, when we are calling ``positive" (resp. ``negative") the ones corresponding to Bessel functions of positive (resp. negative) order. It can then be shown that both positive and negative  ones fall in the domain of $\mathcal{D}_{\K^{n-1}}$. The negative ones, though, correspond to eigenvalues $\mu$ of $\mathcal{D}_{\K^{n-1}}$ such that $|\mu|\leq1/2$ (this can be seen by studying the asymptotic behaviours of the Bessel functions). Finally, negative solutions in the domain of $\mathcal{{D}}_{\K^{n-1}}$ prevent the operator $\mathcal{D}_{\K^{n-1}}$ from being selfadjoint. This is the reason why we need the assumption $|\mu|> 1/2$.
\end{remark}

\subsection{The squaring trick and weighted spinors}\label{squaresec}

We now introduce {\em weighted spinors}, the main goal being transforming the system \eqref{systf} into a system of wave equations on $\mathbb{R}^n$ perturbed by a radial, electric potential, in order to exploit the existing theory to obtain dispersive estimates. This strategy has been already employed in \cite{banduy,danzha} in different contexts (the Schr\"odinger equation on spherically symmetric manifolds and equivariant wave maps respectively) and in the predecessor of this paper, \cite{cacdescurv1}, to deal with the local-in-time case.

Take $\sigma  : \R_+ \rightarrow \R_+$ such that for all $r>0$,
$$
\sigma (r) = \frac{r}{\varphi (r)}
$$ 
where $\varphi(r)$ satisfies the assumptions of Theorem \ref{teo1} and write, for $n\geq3$,
$$
\sigma_n = \sigma^{(n-1)/2}.
$$

\begin{lemma}\label{lem:smoothnessofsigma} The map $\sigma $ prolonged by continuity at $0$ is $\mathcal C^1$ and the map
\[
\frac{\sigma'}{\sigma}
\]
is bounded on $(0,\infty)$.
\end{lemma}

\begin{proof} Indeed, for $r\geq 0$, 
\[
\sigma'(r) = \Big( \frac1{r} - \frac{\varphi'(r)}{\varphi(r)}\Big) \sigma.
\]
The map $\sigma$ at $0$ converges to $1$ and we have, as $r\downarrow0$, writing $a= \frac{\varphi''(0)}{2}$,
\[
\frac{\sigma(r)-1}{r} = \frac1{r} \Big( \frac{r}{r + ar^2 +o(r^2)} -1\Big) \rightarrow - a
\]
hence $\sigma'(0) = -a$. What is more, as $r\rightarrow 0$,
\[
\sigma'(r) = \frac1{\varphi(r)} - \frac{r\varphi'(r)}{\varphi^2} = \frac1{r}(-ar +o(r)) \rightarrow -a.
\]
Finally, since $\sigma \rightarrow 1 $ at $0$ and $\sigma >0$, we deduce that $\frac{\sigma'}{\sigma}$ is continuous on $[0,\infty)$.

Finally,
\[
\frac{\sigma'(r)}{\sigma(r)} = \frac1{r} - \frac{\varphi'}{\varphi}
\]
which ensures its boundedness.
\end{proof}

\begin{lemma}\label{lem:normHs} The multiplication by $\sigma_n$ is an isometry from $L^2_{r}$ to $L^2_{\varphi}$. What is more, the multiplication by $\sigma_{n}$ is an isomorphism from $H^1_r$ to $H^1_\varphi$ that satisfies
\[
 \Big[ 1+  c_\varphi \frac{n-1}{2} \Big]^{-1} \|f\|_{H^1_{r}} \leq \|\sigma_{n}f\|_{H^1_{\varphi}} \leq \Big[ 1+ c_\varphi \frac{n-1}{2}\Big] \|f\|_{H^1_{r}}
\]
with $c_\varphi  = \|\frac{\sigma'}{\sigma}\|_{L^\infty((0,\infty))}$.
In particular, by interpolation, we get, for all $s\in [0,1]$,
\[
 \Big[ 1+  c_\varphi \frac{n-1}{2}\Big]^{-s} \|f\|_{H^s_{r}} \leq \|\sigma_{n}f\|_{H^s_{\varphi}} \leq \Big[ 1+ c_\varphi \frac{n-1}{2} \Big]^s \|f\|_{H^s_{r}}.
\]
\end{lemma}

\begin{proof} The fact that $\sigma_{n}$ is an isometry at the $L^2$-level follows by the definition of the norms, as indeed
\[
\|\sigma_{n}f\|_{L^2_{\varphi}}^2 = \int \sigma_{n}^2 f^2 \varphi^{n-1}dr = \int r^{n-1}f^2 dr = \|f\|_{L^2_{r}}^2.
\]

We now estimate $\|\sigma_{n} f\|_{H^1_{\varphi}}$. We have, by the isometry in $L^2$,
\[
\|\sigma_{n} f\|_{H^1_{\varphi}} \leq  \|f\|_{H^1_{r}} + \|\sigma_{n}' \sigma_{n}^{-1}f\|_{L^2_{r}}.
\]
A direct computation yields
\[
\sigma_{n}' \sigma_{n}^{-1} = \frac{n-1}{2}\frac{\sigma'}{\sigma} .
\]
By H\"{o}lder's inequality, we get
\[
\|\sigma_{n}' \sigma_{n}^{-1}f\|_{L^2_{r}}\leq \frac{n-1}{2} c_\varphi \|f\|_{L^2_{r}}.
\]
We now estimate $\|\sigma_{n}^{-1} g\|_{H^1_{r}}$. We have, by isometry in $L^2$,
\[
\|\sigma_{n}^{-1} g\|_{H^1_{r}} \leq \|g\|_{H^1_{\varphi}} + \|(\sigma_{n}^{-1})'\sigma_{n}g\|_{H^1_{\varphi}}.
\]
We have 
\[
(\sigma_{n}^{-1})'\sigma_{n} = -\frac{\sigma_{n}'}{\sigma_{n}} =- \frac{n-1}{2}\frac{\sigma'}{\sigma}  .
\]
We use H\"{o}lder's inequality to get
\[
\|(\sigma_{n}^{-1})'\sigma_{n}g\|_{H^1_{\varphi}}\leq \frac{n-1}{2} c_\varphi \|g\|_{L^2_{\varphi}}.
\]
This concludes the proof.
\end{proof}

\begin{lemma}We have 
\begin{equation}\label{tildh}
 h_{\mu,n}:=\sigma_{n}^{-1} h_{\mu} \sigma_{n} = \begin{pmatrix} m & -\partial_r -\frac{n-1}{2r} + \frac{\mu}{\varphi} \\ \partial_r +\frac{n-1}{2r} + \frac{\mu}{\varphi} & -m \end{pmatrix}.
\end{equation}
\end{lemma}

\begin{proof} Straightforward computation. \end{proof}

\begin{proposition}\label{prop:continuityandconstants} Let $s \in [-1,1]$, and $p,q \geq 1$. If $e^{-it h_{\mu,n}}$ is a continuous operator from $H_{r}^{1/2}$ to $ L^p(\R, W^{s,q}_{r})$, then $e^{-it h_{\mu}}$ is a continuous operator from $H_{\varphi}^{1/2}$ to  $\sigma_{n}^{1-2/q} L^p(\R, W^{s,q}_{\varphi})$ and 
$$
\big\| e^{-it h_{\mu}} \big\|_{H_{\varphi}^{1/2}\rightarrow \sigma_{n}^{1-2/q} L^p(\R, W^{s,q}_{\varphi})} \leq C_\varphi \big\| e^{-it  h_{\mu,n}} \big\|_{H_{r}^{1/2}\rightarrow  L^p(\R, W^{s,q}_{r})}
$$ 
with a constant $C_\varphi$ that does not depend on $\mu$.
\end{proposition}

This is a consequence of the fact that 
$$
e^{-it h_{\mu}} = \sigma_{n} e^{-it  h_{\mu,n}} \sigma_{n}^{-1},
$$
of Lemma \ref{lem:normHs} and of the following lemma.

\begin{lemma}\label{lemsigma}
 The multiplication by $\sigma_{n}$ is a continuous operator from 
$$
 L^p(\R, W^{s,q}_{r})
$$ 
to 
$$
\sigma_{n}^{1-2/q}L^p(\R,W^{s,q}_{\varphi})
$$
for any $s\in[-1,1]$ and any $q\in (1,\infty)$.
\end{lemma}

\begin{proof} First of all, the norm in the $t$ variable is not relevant in the proof, hence we only prove that the multiplication by $\sigma_n$ is continuous from $W^{s,q}_r$ to $\sigma_n^{1-2/q} W_\varphi^{s,q}$ for $q\in (1,\infty)$ and $s\in [-1,1]$. This is equivalent to proving that the multiplication by $\sigma_n^{2/q}$ is continuous from $W_r^{s,q}$ to $W_\varphi^{s,q}$. 

For non-negative $s$, by interpolation, we can reduce the proof to the cases, $s=0,1$. 

For negative $s$, by duality, the continuity of the multiplication by $\sigma_n^{2/q}$ from $W_r^{s,q}$ to $W_\varphi^{s,q}$ is implied by the continuity of the multiplication by $\sigma_n^{-2/q'}$ from $W_\varphi^{-s,q'}$ to $W_r^{-s,q'}$ where $q'$ is the conjugated exponent of $q$. 

Therefore, it sufficient to prove the following, for all $q\in (1,\infty)$ : \begin{enumerate}
\item the multiplication by $\sigma_n^{2/q}$ is an isometry from $L^q_r$ to $L^q_\varphi$,
\item the mutiplication by $\sigma_n^{2/q}$ is continuous from $W_r^{1,q}$ to $W_\varphi^{1,q}$,
\item the multiplication by $\sigma_n^{-2/q}$ is continuous from $W_\varphi^{1,q}$ to $W_r^{1,q}$.
\end{enumerate}

(1) Let $f\in L^q_r$, we have by definition
\[
\|\sigma_n^{2/q}f\|_{L^q_\varphi}^q = \int_0^\infty \sigma_n^2 \varphi^{n-1}|f|^q 
\]
and using the definition of $\sigma_n$,
\[
\|\sigma_n^{2/q}f\|_{L^q_\varphi}^q = \int_0^\infty r^{n-1}|f(r)|^q dr = \|f\|_{L^q_r}^q. 
\]

(2) From (1), it is sufficient to prove that for all $f\in W_r^{1,q}$, we have 
\[
\|\partial_r (\sigma_n^{2/q} f)\|_{L^q_\varphi} \lesssim \|f\|_{W_r^{1,q}}.
\]
By the Leibniz rule, we have 
\[
\|\partial_r (\sigma_n^{2/q} f)\|_{L^q_\varphi} = \|\sigma_n^{2/q}\Big(\frac{n-1}{q}\frac{\sigma'}{\sigma} f + \partial_r f \Big) \|_{L^q_\varphi}
\]
and from (1), we get
\[
\|\partial_r (\sigma_n^{2/q} f)\|_{L^q_\varphi} = \|\frac{n-1}{q}\frac{\sigma'}{\sigma} f + \partial_r f \|_{L^q_r}.
\]
We conclude by using the fact that $\frac{\sigma'}{\sigma}$ is bounded.

(3) From (1), it is sufficient to prove that
\[
\|\partial_r (\sigma_n^{-2/q} f)\|_{L^q_r} \lesssim \|f\|_{W_\varphi^{1,q}}.
\]
By the Leibniz rule, we have 
\[
\|\partial_r (\sigma_n^{-2/q} f)\|_{L^q_r} = \|\sigma_n^{-2/q}\Big(-\frac{n-1}{q}\frac{\sigma'}{\sigma} f + \partial_r f \Big) \|_{L^q_r}
\]
and from (1), we get
\[
\|\partial_r (\sigma_n^{-2/q} f)\|_{L^q_r} = \|-\frac{n-1}{q}\frac{\sigma'}{\sigma} f + \partial_r f \|_{L^q_\varphi}.
\]
We conclude by using the fact that $\frac{\sigma'}{\sigma}$ is bounded. 
\end{proof}

As recalled in the introduction, the (massless) Dirac operator has been constructed as some square root of the Laplacian; in other words, every solution to the free Dirac equation on $\mathbb{R}^n$ satisfies a system of decoupled free wave/Klein-Gordon equations. This point of view can be carried at the ``radial" level:

\begin{lemma}\label{lem:square} Let $V \in  C^2(0,\infty)$ and let
\[
h_{V,n} = \begin{pmatrix} m & -\Big( \partial_r + \frac{n-1}{2r} \Big) + V \\
\Big( \partial_r + \frac{n-1}{2r} \Big) + V & -m \end{pmatrix}.
\]
Then we have 
\[
h_{V,n}^2 = \begin{pmatrix} m^2 + H_{c_-} & 0 \\ 0 & m^2 + H_{c_+}\end{pmatrix}
\]
with
\begin{equation}\label{defH}
H_{c_\pm} = -\Big( \partial_r^2 + \frac{n-1}{r}\partial_r\Big) +c_{\pm}
\end{equation}
and 
\[
c_{\pm} = -\frac{(n-1)(n-3)}{4r^2} + V^2 \pm \partial_r V.
\]
\end{lemma}

\begin{remark}
In other words, if $v= \begin{pmatrix} v_+ \\ v_- \end{pmatrix}$ solves the  equation 
$$
i\partial_t v = \tilde h_{\mu} v
$$
with initial datum $v_0 = \begin{pmatrix}
v_{0,+} \\ v_{0,-}
\end{pmatrix}$, then $v_+$ and $v_-$ solve respectively
$$
\partial_t^2 v_+ = -m^2v_+ -  H_{c_-}v_+ \quad \textrm{ and }\quad \partial_t^2 v_- = -m^2v_- -H_{c_+} v_-
$$ 
with initial data
$$
\begin{pmatrix} v_{+}(t=0) \\ \partial_t v_+ (t=0) \end{pmatrix} = \begin{pmatrix}
v_{0,+} \\ -i mv_{0,+} + i \Big( \partial_r + \frac1{r} - V \Big) v_{0,-} 
\end{pmatrix} 
$$
and
$$ \quad \begin{pmatrix}
v_-(t=0) \\ \partial_t v_-(t=0) 
\end{pmatrix} = \begin{pmatrix}
v_{0,-} \\ im v_{0,-} -i \Big( \partial_r + \frac1{r} +V \Big) v_{0,+}
\end{pmatrix}.
$$
\end{remark}

\begin{proof} Straightforward computation. \end{proof}

 \section{The wave and Klein-Gordon equation \\ with potentials of critical decay}\label{secstr}

 In this section we review the well-known theory on dispersive estimates for critical perturbations of the wave and Klein-Gordon flows, discussing in particular how the available results can be adapted to deal with our problem. As the strategy and the results below are  classical, we will only sketch most of them, providing  references to fill in the details.

\subsection{General Kato-smoothing}
We start with the following 

\begin{proposition}\label{prop:Hcsupersmooth} Let $n\geq 3$ be the dimension, and let $c \in C^1((0,\infty))$ and $r=|x|$. Assume that
\begin{equation}\label{def_delta_c} 
\delta_c : = \min \Big[ \frac14, \inf \Big( cr^2 + \frac{(n-2)^2}{4}\Big), \inf \Big( -r^3c' -r^2c + \frac{(n-2)^2}{4}\Big)\Big] >0 .
\end{equation}
Then the operator $H_c$ defined in \eqref{defH} is positive on the Hilbert space $L^2(\R^n)$ and the operator $|x|^{-1}$ (from $L^2(\R^n)$ to $L^2(\R^n)$) is $H_c$ super-smooth with constant $\delta_c^{-1}$. 
\end{proposition}

\begin{remark} For definitions of smooth and super-smooth operators, we refer to   \cite[Definition 2.1]{dansmoot}. \end{remark}

\begin{proof} Because we have 
\[
\inf \Big( cr^2 + \frac{(n-2)^2}{4}\Big) \geq \delta_c
\]
we get, for any $v \in C^\infty((0,\infty))$ with compact support
\[
\an{v,H_c v}_{L^2} \geq \an{v,\Big(-\partial_r^2 - \frac{n-1}{r}\Big) v }_{L^2} - \an{v,\frac{(n-1)^2}{4r^2}v}_{L^2} + \delta_c \an{v,r^{-2}v}_{L^2}
\]
and by Hardy's inequality, as we are in dimension $n\geq3$,
\[
\an{v,H_c v}_{L^2} \geq  \delta_c \an{v,r^{-2}v}_{L^2}.
\]
Therefore, $H_c$ is positive.

The fact that $|x|^{-1}$ is $H_c$ super-smooth is a consequence of   \cite[Theorem 3.3]{danzha} with $a=\frac{n-1}{r}$. Indeed, for $v$ in the domain of $|x|^{-1}$, write $f = |x|^{-1}v$; writing $R(\lambda +i\varepsilon)$ the resolvant of $H_c$ with $\varepsilon\neq 0$, we have that $ R(\lambda+i\varepsilon) |x|^{-1}v$ is the solution to
\[
u'' + \frac{n-1}{r}u' + (\lambda +i\varepsilon) u -cu= -f
\]
from which we deduce  
\[
\| \, |x|^{-1}u\|_{L^2(\R^n)} \leq \delta_c^{-1} \|\,|x|f\|_{L^2(\R^n)} = \delta_c^{-1}\|v\|_{L^2(\R^n)}.
\]
We get from this
\[
\an{R(\lambda + i\varepsilon) |x|^{-1}v,|x|^{-1}v}_{L^2} = \an{|x|^{-1}u,v}_{L^2}\leq \| \, |x|^{-1}u\|_{L^2(\R^n)}\|v\|_{L^2(\R^n)}
\]
which yields
\[
\an{R(\lambda + i\varepsilon) |x|^{-1}v,|x|^{-1}v}_{L^2}  \leq \delta_c^{-1}\|v\|_{L^2(\R^n)}^2
\]
and concludes the proof.
\end{proof}

\begin{proposition}\label{prop:localsmoothing} Under the same assumptions as Proposition \ref{prop:Hcsupersmooth}, we have that $|x|^{-1} (H_c+\nu)^{1/4} $ is $\sqrt{H_c+\nu}$ super-smooth for any $\nu\in\mathbb{R}^+$ with constant $C_c^2 = (3+\pi)\delta_c^{-1}$ and in particular, for all $v$ in the domain of $(H_c + \nu)^{1/4}$ we have that
\[
\big \|\, |x|^{-1} e^{-it\sqrt{\nu + H_c}}v\big\|_{L^2(\R\times \R^n)} \leq C_c \|(H_c+\nu)^{1/4} v\|_{L^2(\R^n)}.
\]
\end{proposition}

\begin{proof} This is a direct consequence of  \cite[Theorem 2.4]{dansmoot}.
\end{proof}

\begin{proposition}\label{prop:locsmoothwave} Under the assumptions of Proposition \ref{prop:Hcsupersmooth} and assuming $\nu$ nonnegative, set $U_c(t)$ to be the flow of the equation
\begin{equation}\label{Cauchyprobwave}
\left \lbrace{\begin{array}{cc}
\partial_t^2 u + H_c u + \nu u = 0, & \\
u(t=0) = u_0, & \partial_t u(t=0) = u_1.
\end{array}} \right.
\end{equation}
Set also $\mathcal X_{c,\nu}$ and $\mathcal H^{1/2}_{c,\nu}$ to be spaces respectively induced by the norms
\[
\|(f,g)\|_{\mathcal X_{c,\nu}}^2 = \|\,|x|^{-1} f\|_{L^2(\R\times \R^n)}^2 + \|\,|x|^{-1}(H_c+\nu)^{-1/2}g\|_{L^2(\R\times \R^n)}^2
\]
and 
\[
\|(u_0,u_1)\|_{\mathcal H^{1/2}_{c,\nu}}^2 = \|(H_c +\nu)^{1/4}u_0\|_{L^2(\R^n)}^2 + \|(H_c+\nu)^{-1/4}u_1\|_{L^2(\R^n)}^2  .
\]
Then we have for all $(u_0,u_1) \in \mathcal H^{1/2}_{c,\nu}$,
\[
\|U_c(t)(u_0,u_1)\|_{\mathcal X_{c,\nu}} \leq C_c \|(u_0,u_1)\|_{\mathcal H^{1/2}_{c,\nu}}.
\]
\end{proposition}

\begin{proof} The proof follows the usual lines assuming, without loss of generality, that $u$ is real and use the transform
\[
U = u+ i(H_c+\nu)^{-1/2} \partial_t u.
\]
\end{proof}

\subsection{Application to the Dirac equation with critical potentials}

\begin{proposition}\label{prop:locsmoothdirac} Let $V \in C^2((0,\infty))$. Write
 $c_\pm = - \frac{(n-1)(n-3)}{4r^2} + V^2 \pm V'$. Set $S_{V,n}(t)$ to be the flow of equation $i\partial_t u= h_{V,n}u$ with $h_{V,n}$ as in Lemma \ref{lem:square} . Assume that
\begin{equation}\label{defdeltaV}
\delta_V^\pm : = \min \Big[ \frac14, \inf \Big( \frac14 + r^2(V^2 \pm V') \Big),\inf \Big( \frac14  - r^3 (2VV' \pm V'') - r^2(V^2\pm V') \Big) \Big] >0.
\end{equation}

Then we have that for all $u_0 \in \mathcal H^{1/2}_{c_+,c_-,m}$ the solution $S_{V,n}(t) u_0$ satisfies \[
\||x|^{-1} S_{V,n}(t) u_0 \|_{L^2(\R \times \R^n)} \leq 3(C_{c_+} + C_{c_-}) \|u_0\|_{\mathcal H^{1/2}_{c_+,c_-,m}},
\]
where $\mathcal H^{1/2}_{c_+,c_-,m}$ is the space induced by the norm 
\[
\left\Vert \begin{pmatrix}
f\\g 
\end{pmatrix}\right\Vert_{\mathcal H^{1/2}_{c_+,c_-,m}} = \|(m^2 + H_{c_-}  )^{1/4} f\|_{L^2(\R^n)} + \|(m^2 + H_{c_+})^{1/4}g\|_{L^2(\R^n)}.
\]
Finally, we have $\delta_{c_\pm} = \delta_V^{\pm}$.
\end{proposition}

\begin{proof} Set 
\[
u_0 = \begin{pmatrix} f_0 \\ g_0 \end{pmatrix} \quad \textrm{and}\quad S_{V,n}(t)(u_0) = \begin{pmatrix} f\\ g \end{pmatrix}.
\]
Write $V_\pm = V \pm \Big( \partial_r  + \frac{n-1}{2r}\Big)$. By a straightforward computation we get that $V_-V_+ = H_{c_-}$ and $V_+V_- = H_{c_+}$. Therefore, we recall that
\[
h_{V,n}^2 = \begin{pmatrix} m^2 +  H_{c_-} & 0 \\ 0 & m^2 + H_{c_+} \end{pmatrix}.
\]
and that $f$ is the solution to $\partial_t^2 f + H_{c_-}f + m^2 f = 0$ with initial datum
\[
f(t=0) = f_0\quad \textrm{and} \quad\partial_t f(t=0) = f_1: = -imf_0 -iV_-g_0.
\]
A direct computation yields $\delta_V^\pm = \delta_{c_\pm}$ and thus
\[
\|\,|x|^{-1}f\|_{L^2_{t,x}} \leq C_{c_-} \Big( \|(m^2 + H_{c_-})^{1/4}f_0\|_{L^2_x} + \|(m^2 + H_{c_-})^{-1/4}f_1\|_{L^2_x}\Big)
\]
We have that
\[
 \|(m^2 + H_{c_-})^{-1/4}mf_0\|_{L^2} \leq \sqrt{|m|} \|f_0\|_{L^2}
 \]
since $m^2 + H_{c_-} \geq m^2$. 

What is more
\[
\|(m^2 + H_{c_-})^{-1/4}V_-g_0\|_{L^2}^2 = \an{(m^2 + H_{c_-})^{-1/4}V_-g_0,(m^2 + H_{c_-})^{-1/4}V_-g_0}_{L^2}.
\]
Since $m^2 + H_{c_-} \geq H_{c_-}$, we have 
\[
\|(m^2 + H_{c_-})^{-1/4}V_-g_0\|_{L^2}^2 \leq \an{ H_{c_-}^{-1/4}V_-g_0,  H_{c_-}^{-1/4}V_-g_0}_{L^2}.
\]

By taking adjoints, we get
\[ 
\|(m^2 + H_{c_-})^{-1/4}V_-g_0\|_{L^2}^2 \leq \an{V_+ H_{c_-}^{-1/2}V_-g_0,g_0}_{L^2}.
\]
Let $A = V_+ H_{c_-}^{-1/2}V_-$. The operator $A$ is positive and 
\[
A^2  = V_+H_{c_-}^{-1/2}V_-V_+  H_{c_-}^{-1/2} V_-.
\]
Since $V_-V_+ = H_{c_-}$ we get
\[
A^2 = V_+ H_{c_-}^{-1/2}H_{c_-}  H_{c_-}^{-1/2} V_- = V_+V_- = H_{c_+}.
\]
Finally,
\[
\|\,|x|^{-1}f\|_{L^2} \leq C_{c_-} \Big( 2\|(m^2 + H_{c_-})^{1/4}f_0\|_{L^2} + \|(m^2 + H_{c_+})^{1/4}g_0\|_{L^2}\Big).
\]

With a similar computation, we get
\[
\|\,|x|^{-1}g\|_{L^2} \leq C_{c_+} \Big( 2\|(m^2 + H_{c_+})^{1/4}g_0\|_{L^2} + \|(m^2 + H_{c_-})^{1/4}f_0\|_{L^2}\Big).
\]
\end{proof}

We can now exploit the powerful Rodnianski-Schlag argument (see \cite{rodsch}) to deduce Strichartz estimates from Proposition \ref{prop:locsmoothdirac}.

\begin{proposition}\label{prop:strichartz} Assume that $(p,q,m)$ is admissible, as in Definition \ref{def:admP}. 
Then there exists a constant $C = C(p,q,m)$ such that for all $u_0 \in H^{1/2}_r \cap \mathcal H^{1/2}_{c_+,c_-,m}$, we have 
\begin{multline}\label{strichartz}
\| S_{V,n}(t)u_0\|_{L^p,W^{1/q-1/p,q}_r} \leq C \Big( (1+\|rV\|_{L^\infty((0,\infty))})\|u_0\|_{H^{1/2}}\\ + (\|r^2c_+\|_{L^\infty((0,\infty))} +  \|r^2c_-\|_{L^\infty((0,\infty))})((\delta_V^+)^{-1}+(\delta_V^-)^{-1})\|u_0\|_{\mathcal H^{1/2}_{c_+,c_-,m}}\Big).
\end{multline}
\end{proposition}

\begin{proof} For $c\in  C^1$, write $V_c(t)$ the projection on the first coordinate of the flow of 
\[
\left\lbrace{\begin{array}{c}
\partial_t^2u + m^2u+ H_c u = 0\\
u(t=0) =u_0, \partial_t u(t=0) = u_1 \end{array}}\right.
\]
Set
\[
\begin{pmatrix} f \\ g \end{pmatrix} = S_{V,n}(t) u_0\quad \textrm{and} \quad \begin{pmatrix} f_0 \\ g_0 \end{pmatrix} =  u_0.
\]

We have 
\[
f = V_{c_-}(t) (f_0,f_1) = V_0(t) (f_0,f_1) - \int_{0}^t \frac{\sin((m^2+H_0)^{1/2}(t-\tau)}{(m^2+H_0)^{1/2}}(c_- f(\tau))d\tau
\]
with $f_1 = -imf_0 - i(V - (\partial_r + \frac{n-1}{2r}))g_0$. 

We have, by the mixed Strichartz-local smoothing estimates for $V_0$ (the flow of the free wave/Klein-Gordon equation that does not depend on $V$)
\[
\|f\|_{L^p,W^{s,q}_r} \leq C \Big( \|f_0\|_{H^{1/2}_r} + \|f_1\|_{H^{-1/2}_r} + \|r c_- f\|_{L^2(\R\times \R^n)}\Big).
\]
Because of the Hardy inequality, we have
\[
 \|f_1\|_{H^{-1/2}_r} \lesssim \|f_0\|_{L^{2}} + (1+\|rV\|_{L^\infty}) \|g_0\|_{H^{1/2}_r}.
\]

Besides,
\[
\|r c_- f\|_{L^2(\R\times \R^n)}\leq \|r^2c_-\|_{L^\infty} \|r^{-1}f\|_{L^2(\R\times \R^n)}.
\]
Because of local smoothing on $S_{V,n}(t)$, we get
\[
\|r^{-1}f\|_{L^2(\R\times \R^n)} \lesssim (C_{c_+} + C_{c_-}) \|u_0\|_{\mathcal H^{1/2}_{c_+,c_-,m}}.
\]
A similar computation on $g$ yields the result.
\end{proof}

\subsection{Application to the Dirac equation in curved manifolds}

In this section, we set $V_\mu = \frac{\mu}{\varphi}$, $\delta_\pm(\mu) = \delta_{V_\mu}^\pm$ and 
\[
c_{\pm}(\mu) = -\frac{(n-1)(n-3)}{4r^2} + V_\mu^2 \pm V_\mu' = -\frac{(n-1)(n-3)}{4r^2} + \mu \frac{\mu \mp \varphi' }{\varphi^2}.
\]
Finally, we set $H_\pm (\mu) = H_{c_{\pm}}(\mu)$.

\begin{lemma}\label{lemnorms}
For any $|s|\leq1$ we have the bound
$$
\|(m^2 + H_\pm(\mu))^{s/2}v\|_{L^2(\R^n)}\lesssim_{\varphi,m} (1+\mu^2)^{s/2} \|v\|_{ H^s(\R^n)}
$$ 
In particular, we have 
\[
 \|u_0\|_{\mathcal H^{1/2}_{c_+(\mu),c_-(\mu),m}} \lesssim_{\varphi,m} \sqrt{|\mu|} \|u_0\|_{H^{1/2}_{r,n}}
\]
\end{lemma}
\begin{remark}
Differently from \cite{danzha}, we need to keep track on the dependence on $\mu$ of the inequalities above: indeed, if on one hand for the purpose of Theorem \ref{teo1} such a dependence is irrelevant (as it is in \cite{danzha}), in view of Theorem \ref{teo2} it will play an important role, as powers of $\mu$ will be traded with angular derivatives on the initial data. 
\end{remark}
\begin{proof}
We have that, for any $\mu$, 
$$
 c_\pm(\mu)(r)\leq \frac{C_\varphi \mu^2}{r^2}
$$
with 
\[
C_\varphi = \max \Big( \| \frac{r^2}{\varphi^2}\|_{L^\infty((0,\infty))}, 2\|\frac{r^2 \varphi '}{\varphi^2}\|_{L^\infty((0,\infty))}\Big).
\]
As done in \cite[Section 2]{danzha}, the result follows from the application of Hardy inequality and interpolation in a standard way. We omit the details.
\end{proof}

We deduce from Proposition \ref{prop:strichartz} the following result.

\begin{proposition}\label{prop:strichartzDiracPerturbed} Assume that $(p,q,m)$ is admissible, as in Definition \ref{def:admP}. Then there exists a constant $C = C(p,q,m,\varphi)$ such that for all $u_0 \in H^{1/2}_r $, we have
\begin{equation}\label{strichartzDiracPert}
\| S_{V_\mu,n}(t)u_0\|_{L^p,W^{1/q -1/p,q}_r} \leq C   |\mu|^{5/2}((\delta_{V_\mu}^+)^{-1/2}+(\delta_{V_\mu}^-)^{-1/2})\|u_0\|_{H^{1/2}_r}.
\end{equation}
\end{proposition}

\begin{proof} Estimate \eqref{strichartzDiracPert} is a direct consequences of Proposition \ref{prop:strichartz}, Lemma \ref{lemnorms}, and the assumptions \[
\|rV\|_{L^\infty} \lesssim_\varphi |\mu|, \quad \|r^2c_\pm\|_{L^\infty} \lesssim_\varphi \mu^2.
\]
\end{proof}

By interpolation, we get the following
\begin{corollary}Assume that $(p,q,m)$ is admissible, as in Definition \ref{def:admP} and let $\varepsilon>0$ if $m=0$ in $3d$, otherwise $\varepsilon=0$. There exists $C = C(p,q,m,\varphi,\varepsilon)$ such that for all $u_0 \in H^{1/2}_r $, we have 
\begin{equation}\label{strichartzcor}
\| S_{V,n}(t)u_0\|_{L^p,W^{1/q-1/p,q}_r} \leq C   |\mu|^{5/p+\varepsilon}((\delta_V^+)^{-1/2}+(\delta_V^-)^{-1/2})^{2/p + \varepsilon}\|u_0\|_{H^{1/2}_r}
\end{equation}

\end{corollary}

\begin{proof}
If $\theta:=\frac2p+\varepsilon>1$, then the result is a consequence of estimate \eqref{strichartzDiracPert}. Otherwise, we obtain \eqref{strichartzcor} by interpolating \eqref{strichartzDiracPert} with the standard $L^\infty H^s$ estimate. Notice that the assumption $\varepsilon>0$ is needed if the endpoint couple is not admissible: we refer to the proofs of Lemmas 5.1 and 5.5 in \cite{cacdescurv1}.
\end{proof}

Exploiting Proposition \ref{prop:continuityandconstants}, we eventually get Theorem \ref{teo1}.


\section{Strichartz estimates in the asymptotically flat case}\label{sec:strich}

In this section we specialize to the ``asymptotically flat" case. First of all, we provide a slightly more precise version of Assumptions \textbf{(A2)} and in particular of the constant $C$. As a consequence, we are able to give some explicit condition in order for hypothesis \eqref{crucialcond} to be satisfied. Then, after further restricting to the case $\K^{n-1}=\mathbb{S}^{n-1}$, we prove Theorem \ref{teo2}.

\subsection{Assumptions}

Let us assume that the infimum of the positive part of the spectrum of the Dirac operator on $\K^{n-1}$, denoted by $\mu_0$, is strictly bigger than $\frac12$, and that $\varphi$ is asymptotically flat, in other words, that
\[
\varphi = r(1 + \varphi_1)
\]
with the following assumption on $\varphi_1$ : \begin{itemize}
\item $\varphi_1$ is non-negative and bounded,
\item that $A_\varphi = \|\varphi_1 + r\varphi_1'\|_\infty$ and 
\[
B_\varphi = \|r\varphi_1' + (1+\varphi_1)(\varphi_1 + r\varphi_1')\|_\infty + \|2r^2 (\varphi_1')^2 + (1+\varphi_1)r^2\varphi_1''\|_\infty
\]
are well-defined,
\item that
\[
\max(A_\varphi,B_\varphi) \left\lbrace{\begin{array}{cc} \leq  1 & \textrm{if } \mu_0 \geq 2\\
< \min(\frac14 + \mu_0^2 -\mu_0 , \frac18) & \textrm{otherwise}\end{array}} \right. .
\]
\end{itemize}

\subsection{Asymptotically flat manifolds are admissible}\label{admissasflat}

In this subsection we prove Proposition \ref{verification}: if $\varphi(r)$ satisfies assumptions above, condition \eqref{crucialcond} is satisfied, and therefore the Strichartz estimates proved in Theorem \ref{teo1} hold. The only thing we need to prove is the following.

\begin{lemma}\label{deltalemma}
 Under the above assumptions on $\varphi_1$, we have for all $\mu \geq \mu_0$,
\[
\delta_\pm(\mu) \geq \left\lbrace{\begin{array}{cc} 
\frac14 & \textrm{if }\mu_0 \geq 2,\\
 \min(\frac14 + \mu_0^2 -\mu_0 , \frac18) - \max (A_\varphi,B_\varphi) & \textrm{otherwise.}
 \end{array}} \right. .
\]
\end{lemma}

\begin{proof} 
We have 
\[
I(r):= \frac14 + r^2(V_{\mu} \pm V_\mu') = \frac14 + \frac{\mu^2\mp \mu}{(1+\varphi_1)^2} \mp \mu\frac{\varphi_1 + r\varphi_1'}{(1+\varphi_1)^2}.
\]
Therefore,
\[
I(r) \geq \frac14 +  \frac{\mu^2-\mu}{(1+\varphi_1)^2} -\mu \frac{A_\varphi}{(1+\varphi_1)^2}.
\]
Case 1 : $\mu \geq 2$, we have since $\mu^2 - \mu \geq \mu$,
\[
I(r) \geq \frac14 + \frac{\mu}{(1+\varphi_1)^2} (1-A_\mu)
\]
and since $A_\mu \leq 1$, we have $I(r)\geq \frac14$. 

Case 2 : $\mu \in [1,2)$, we have since $\varphi_1\geq 0$ and $\mu^2 - \mu \geq 0$,
\[
I(r) \geq \frac14 - 2 A_\varphi \geq \frac18 - A_\varphi
\]
which is positive since $A_\varphi < \frac18$. 

Case 3 : $\mu \in [\mu_0,1)$, we have, since $\varphi_1 > 0$ and $\mu^2 - \mu \geq \mu_0^2 - \mu_0$,
\[
I(r) \geq \frac14 +\mu_0^2 - \mu_0  - A_\varphi
\]
which is positive.

Set
\[
Q_\pm(r) = \frac14 - r^3(2V_\mu V_\mu' \pm V_\mu'') -r^2 (V_\mu^2\pm V_\mu').
\]
We have 
\[
Q_\pm(r) = \frac14 + \frac{\mu^2 \mp \mu}{(1+\varphi_1)^2} + \frac{\mu^2 \mp \mu}{(1+\varphi_1)^3}f(r) \mp \frac{\mu}{(1+\varphi_1)^3}g(r)
\]
with
\[
f(r) = r\varphi_1' + (1+\varphi_1)(\varphi_1 + r\varphi_1')\quad \textrm{and}\quad g(r) = 2r^2 (\varphi_1')^2 + (1+\varphi_1) r^2 \varphi_1''.
\]

Case 1 : We consider $Q_+ (r)$. We have 
\[
Q_+(r) = \frac14 + \frac{\mu^2 + \mu}{(1+\varphi_1)^2} + \frac{\mu^2 + \mu}{(1+\varphi_1)^3}f(r) + \frac{\mu}{(1+\varphi_1)^3}g(r)
\]
hence
\[
Q_+(r) \geq \frac14 + \frac{\mu^2 + \mu}{(1+\varphi_1)^2} ( 1 - B_\varphi)
\]
and since $B_\varphi \leq 1$, we have $Q_+(r) \geq \frac14$.

Case 2 : we consider $Q_-(r)$. We have 
\[
Q_-(r) = \frac14 + \frac{\mu^2 -\mu}{(1+\varphi_1)^2} + \frac{\mu^2 - \mu}{(1+\varphi_1)^3}f(r) - \frac{\mu}{(1+\varphi_1)^3}g(r).
\]

Case 2.1 : $\mu \geq 2$, we have $\mu^2 - \mu \geq \mu$, hence 
\[
Q_-(r) \geq \frac14 + \frac{\mu^2-\mu}{(1+\varphi_1)^2} (1-B_\varphi)
\]
and since $B_\varphi \leq 1$, we have $Q_-(r)\geq \frac14$.

Case 2.2 : $\mu \in [1,2)$. We have $\mu^2 - \mu \leq 2$ and $\mu \leq 2$, hence
\[
Q_-(r) \geq \frac14 -2 B_\varphi.
\]

Finally, case 2.3 : $\mu \in [\mu_0,1)$, we have $0>\mu^2-\mu \leq \mu_0^2 - \mu_0$ and $|\mu^2 - \mu|\leq 1$ hence
\[
Q_-(r) \geq \frac14 +\mu_0^2 - \mu_0  - B_\varphi
\]
which concludes the proof.
\end{proof}

\subsection{Local smoothing in the asymptotically flat case}

From this subsection, we assume that $\K^{n-1}$ is $\mathbb S^{n-1}$. We have that the positive spectrum of the Dirac operator on the sphere is $\frac{n-1}{2} + \N$. We see hence that in dimension higher that $5$, we have that $\delta_\pm (\mu) \geq \frac14$ for all $\mu$ in the spectrum. In any case, for a fixed $\varphi$ satisfying Assumptions \textbf{A2}, we have that $\delta$ is uniformly bounded in $\mu$ by below. 

For $\mu$ in the spectrum of the Dirac operator on the sphere, we write $\mathcal H_\mu$ the space generated by 
\[
\left\{ \begin{pmatrix}
(1+i\tilde \alpha^0) \psi_\mu \\0
\end{pmatrix}, \begin{pmatrix}
0 \\ (1-i\tilde \alpha^0) \psi_\mu
\end{pmatrix}, \quad \mathcal D_{\mathbb S^{n-1}} \psi_\mu = \mu \psi_\mu \right\}.
\]

For $0\leq a <b $ we set
\[
\mathcal H_{a,b} = \bigoplus_{|\mu|\in [a,b]} \mathcal H_\mu 
\]
and $p_{a,b}$ the orthogonal projection onto $L^2_r\otimes \mathcal H_{a,b}$. 

We recall that $\sigma_n^{-1} \mathcal D_\Sigma \sigma_n$ is entirely described by the $h_{_\mu,n}$ and thus commute with $p_{a,b}$. We write $S_n(t)$ the flow of 
\[
i\partial_t - \sigma_n^{-1} \mathcal D_\Sigma \sigma_n = 0 .  
\]
We deduce the following proposition.

\begin{proposition}\label{prop:localsmoothaflat} Let $u_0 \in H^{1/2}(\R^n)$, we have 
\[
\|r^{-1}p_{a,b}S_n(t)u_0\|_{L^2(\R\times \R^n)} \lesssim_{m,\varphi,n} b^{1/2} \|p_{a,b} u_0\|_{H^{1/2}(\R^n)}.
\]
\end{proposition}

\begin{proof} Let $u_{0,\mu}$ to be the orthogonal projection of $u_0$ over $L^2_r\otimes \mathcal H_\mu$ and $u_\mu = S_n(t)u_{0,\mu}$. Because the orthogonal projection over $L^2_r\otimes \mathcal H_\mu$ and $S_n(t)$ commute, we get
\[
\|r^{-1}p_{a,b}S_n(t)u_0\|_{L^2(\R\times \R^n)}^2 = \sum_{|\mu| \in [a,b]} \|r^{-1}u_\mu\|_{L^2(\R\times \R^n)}^2.
\]

From Proposition \ref{prop:locsmoothdirac}, we have 
\[
 \|r^{-1}u_\mu\|_{L^2(\R\times \R^n)} \leq 3(C_{c_+(\mu)} + C_{c_-(\mu)})\|u_{0,\mu}\|_{\mathcal H_{c_+(\mu),c_-(\mu),m}}
\]
where, by abuse of notation, we identified $u_{0,\mu}$ with
\[
\sum_{j} \frac{f_j}{\sqrt 2}  \begin{pmatrix}
(1+i\tilde \alpha^0) \psi_{\mu,j} \\0
\end{pmatrix}+ \frac{g_j}{\sqrt 2} \begin{pmatrix}
0 \\ (1-i\tilde \alpha^0) \psi_{\mu,j}
\end{pmatrix}
\]
where the (finite) family $(\psi_{\mu,j})_j$ is an orthonormal basis of the eigenspace of $\mathcal D_{\mathbb S^{n-1}}$ associated to $\mu$, and we identified $\|u_{0,\mu}\|_{\mathcal H_{c_+(\mu),c_-(\mu),m}}^2$ with
\[
\sum_j \big \| \begin{pmatrix}
f_j \\g_j
\end{pmatrix}\big\|_{\mathcal H_{c_+(\mu),c_-(\mu),m}}^2.
\]
From Lemma \ref{lemnorms}, we have for all $j$,
\[
\big \| \begin{pmatrix}
f_j \\g_j
\end{pmatrix}\big\|_{\mathcal H_{c_+(\mu),c_-(\mu),m}} \lesssim_{m,\varphi} \sqrt{|\mu|} \big \| \begin{pmatrix}
f_j \\g_j
\end{pmatrix}\big\|_{H^{1/2}(\R^n)}
\]
from which we deduce
\[
\|u_{0,\mu}\|_{\mathcal H_{c_+(\mu),c_-(\mu),m}} \lesssim_{m,\varphi} \sqrt{|\mu|} \|u_{0,\mu}\|_{H^{1/2}(\R^n)} \leq \sqrt b \|u_{0,\mu}\|_{H^{1/2}(\R^n)}.
\]
We conclude by using the fact that $C_{c_+(\mu)}$ and $C_{c_-(\mu)}$ are uniformly bounded in $\mu$.
\end{proof}

\subsection{Restricted Strichartz estimates in the asymptotically flat case}

In this subsection, we prove the following proposition. 

\begin{proposition}\label{prop:strichartrestricted} Let $0\leq a < b$ and let $m,p,q$  be admissible. We have, for all $u_0 \in H^{1/2}(\R^n)$ and all $\varepsilon >0$, 
\[
\|p_{a,b} S_n(t) u_0\|_{L^p(\R, W^{s,q}(\R^n))}\left \lbrace{\begin{array}{cc}
 \lesssim_{m,\varphi, n,\varepsilon,p,q} b^{5/p+\varepsilon} \|p_{a,b}u_0\|_{H^{1/2}} & \textrm{if } n=3,m=0\\
 \lesssim_{m,\varphi, n,p,q} b^{5/p} \|p_{a,b}u_0\|_{H^{1/2}} & \textrm{otherwise },\end{array}} \right.
\]
with $s = \frac1{q}-\frac1{p}$. 
\end{proposition}

\begin{proof} We prove that 
\[
\|p_{a,b} S_n(t) u_0\|_{L^p(\R, W^{s,q}(\R^n))} \lesssim_{m,\varphi, n,} b^{5/2} \|p_{a,b}u_0\|_{H^{1/2}}
\]
for all admissible triplets $(m,p,q)$ and conclude by interpolation.

First, we have 
\[
\sigma_n^{-1} \mathcal D_\Sigma \sigma_n = \mathcal D_{\R^n} + \mathcal V
\]
with $\mathcal V$ the operator
\[
\mathcal V = \Big(\frac1{\varphi} - \frac1{r}\Big) \begin{pmatrix}
0 & \mathcal D_{\mathbb S^{n-1}} \\ \mathcal D_{\mathbb S^{n-1}} & 0
\end{pmatrix}.
\]

Writing $u  = S_n(t) u_0$, we get that $u$ satisfies 
\[
\partial_t^2  u + (\mathcal D_{\R^n} + \mathcal V)^2 u = 0
\]
with initial data $u(t=0) = u_0$ and $\partial_t u(t=0) =: u_1 = -i(\mathcal D_{\R^n} + \mathcal V)u_0$.

We have 
\[
(\mathcal D_{\R^n} + \mathcal V)^2 = \mathcal D_{\R^n}^2 + \mathcal W = m^2 - \lap_{\R^n} +\mathcal W
\]
with
\[
\mathcal W = \{\mathcal V,\mathcal D_{\R^n}\} + \mathcal V^2.
\]

By the Rodnianski-Schlag argument that we previously used, we get
\[
\|p_{a,b} u\|_{L^p,W^{s,q}(\R^n)} \lesssim_{n,p,q} \|p_{a,b}u_0\|_{H^{1/2}(\R^n)} + \|p_{a,b}u_1\|_{H^{-1/2}(\R^n)} + \|r\mathcal W p_{a,b} u\|_{L^2(\R^{n+1})}.
\]

By the commutativity of $p_{a,b}$ and $\mathcal D_{\R^n} + \mathcal V$ we get
\[
\|p_{a,b}u_1\|_{H^{-1/2}} = \|(\mathcal D_{\R^n} + \mathcal V)p_{a,b} u_0\|_{H^{-1/2}}
\]
and since $r(\frac1{\varphi} - \frac1{r})$ is bounded, by Hardy's inequality, we get
\[
\|p_{a,b}u_1\|_{H^{-1/2}} \lesssim_{n,\varphi}\|p_{a,b} u_0\|_{H^{1/2}}.
\]

For the other part, we use that
\[
\|r\mathcal W p_{a,b} u\|_{L^2(\R^{n+1})} \leq \|p_{a,b}r\mathcal W r p_{a,b}\|_{L^2\rightarrow L^2}\|r^{-1}p_{a,b}u\|_{L^2(\R^{n+1})}.
\]
It remains to use Proposition \ref{prop:localsmoothaflat} and prove that $p_{a,b}r\mathcal W r p_{a,b}$ is a bounded operator from $L^2(\R^{n+1})$ to itself and compute the dependence of its norm in $a,b$ to conclude.

Because the multliplication by a radial function and the Dirac operator on the sphere commute, we get that
\[
p_{a,b}r\mathcal V^2 r p_{a,b} = \Big(\frac{r}{\varphi}-1\Big)^2\begin{pmatrix}
p_{a,b}\mathcal D_{\mathbb S^{n-1}}^2 p_{a,b} & 0 \\
0 & p_{a,b}\mathcal D_{\mathbb S^{n-1}}^2 p_{a,b}\end{pmatrix}
\]
and we deduce 
\[
\| p_{a,b}r\mathcal V^2 r p_{a,b}\|_{L^2\rightarrow L^2} \leq \|\Big(\frac{r}{\varphi}-1\Big)^2\|_\infty b^2
\]
which is finite because of the assumptions on $\varphi$. 

What is more, we have 
\[
\mathcal D_{\R^n} = \begin{pmatrix}
m & i\tilde \alpha^0  \Big( \partial_r + \frac{n-1}{2r}\Big) + \frac1{r}\mathcal D_{\mathbb S^{n-1}} \\
i\tilde \alpha^0  \Big( \partial_r + \frac{n-1}{2r}\Big) + \frac1{r}\mathcal D_{\mathbb S^{n-1}} & -m
\end{pmatrix}
\]
We deduce 
\[
\{ \mathcal D_{\R^n}, \mathcal V\} =
\begin{pmatrix}
\mathcal L  & 0 \\ 0 & \mathcal L \end{pmatrix}
\]
with 
\[
\mathcal L = 
\{ i\tilde \alpha^0 \Big( \partial_r + \frac{n-1}{2r}\Big)+ \frac1{r}\mathcal D_{\mathbb S^{n-1}},\Big( \frac1{\varphi} - \frac1{r}\Big)\mathcal D_{\mathbb S^{n-1}}\} 
\]
We have that $i\tilde \alpha^0$ and $\mathcal D_{\mathbb S^{n-1}}$ anticommute, that $\partial_r$ and  $\mathcal D_{\mathbb S^{n-1}}$ commute, and that the multiplication by a radial function commutes with $\mathcal D_{\mathbb S^{n-1}}$. Hence we get
\[ 
\mathcal L = i\tilde \alpha^0 \mathcal D_{\mathbb S^{n-1}}\Big[\partial_r + \frac{n-1}{2r}, \frac1{\varphi} - \frac1{r}\Big] + 2 \Big( \frac1{\varphi} - \frac1{r}\Big)\frac1{r} \mathcal D_{\mathbb S^{n-1}}^2.
\]
We deduce 
\[
p_{a,b}r\{\mathcal V, \mathcal D_{\R^n}\}rp_{a,b} =\begin{pmatrix} \mathcal L_{a,b} & 0 \\ 0 & \mathcal L_{a,b}\end{pmatrix}
\]
with
\[
\mathcal L_{a,b} = ip_{a,b} \tilde \alpha^0\mathcal D_{\mathbb S^{n-1}}p_{a,b}r^2 \partial_r \Big( \frac1{\varphi} - \frac1{r}\Big) + 2 r \Big( \frac1{\varphi} - \frac1{r}\Big)p_{a,b}\mathcal D_{\mathbb S^{n-1}}^2p_{a,b}.
\]
Because 
\[
r^2 \partial_r \Big( \frac1{\varphi} - \frac1{r}\Big) =  \frac{\varphi_1}{1+\varphi_1}- \frac{r\varphi_1'}{(1+\varphi_1)^2}
\]
belongs to $L^\infty$, and so does $r\Big( \frac1{\varphi} - \frac1{r}\Big) = \frac1{1+\varphi_1} - 1$, we get
\[
\|p_{a,b}r\{\mathcal V, \mathcal D_{\R^n}\}rp_{a,b}\|_{L^2\rightarrow L^2} \lesssim_{\varphi} b^2.
\]
This concludes the proof.
\end{proof}

\subsection{Set-up for the Littlewood-Paley argument}

In this subsection, we draw a link between the spherical harmonics and the eigenfunctions of the Dirac operator on the sphere.

\begin{proposition}\label{prop:LPsetup}
Let $\pi_j$ be the orthogonal projection on $\mathcal S_j \otimes L^2_r \otimes \C^M$ where $\mathcal S_j$ are the spherical harmonics of degree in $[2^j,2^{j+1})$ and let $u \in L^2(\R^n,\C^M)$, we have 
\[
\pi_j u = \pi_j p_{a_j,b_j} u
\]
with $a_j = \frac{n-1}{2} + 2^j - 1$ and $b_j = \frac{n-1}2 + 2^{j+1}$.
\end{proposition}

Before proving this proposition, we prove the following short lemma.
\begin{lemma}\label{lem:LPsetup} We have for all $\mu$ in the spectrum of the Dirac operator on the sphere $\mathbb S^{n-1}$,
\[
\mathcal H_\mu \subseteq (\mathcal S_{|\mu|-\frac{n-1}{2}} \oplus \mathcal S_{|\mu| - \frac{n-1}{2} + 1} ) \otimes \C^M.
\]
\end{lemma}

\begin{proof} Let $\psi_\mu$ be an eigenfunction of $\mathcal D_{\mathbb S^{n-1}}$ with eigenvalue $\mu$ and write 
\[
\Psi_\mu^+ = \begin{pmatrix} (1+i\tilde \alpha^0)\psi_\mu \\ 0 \end{pmatrix},\quad \Psi_\mu^-= \begin{pmatrix}
0 \\ (1-i\tilde \alpha^0) \psi_\mu 
\end{pmatrix}.
\]
We have 
\[
\mathcal D_{\R^n}\Psi_\mu^+ = \Big(\mu- \frac{n-1}{2}\Big) \frac1{r} \Psi_\mu^-
\]
and thus
\[
-\lap_{\R^n}\Psi_\mu^+ = \mathcal D_{\R^n}^2 \Psi_\mu^+ = \mathcal D_{\R^n} \Big(\mu- \frac{n-1}{2}\Big) \frac1{r} \Psi_\mu^- = \Big( \mu - \frac{n-1}{2}\Big)\Big( \mu + \frac{n-1}{2}-1\Big) \frac1{r^2} \Psi_\mu^+.
\]
Because $\Psi_\mu^+$ does not depend on $r$, we deduce that it is a spherical harmonics of degree $\mu -\frac{n-1}{2}$ if $\mu >0$ and $-\mu-\frac{n-1}{2}+1$ otherwise. 

The same type of computation yields 
\[
-\lap_{\R^n} \Psi_\mu^- = \Big( \mu+ \frac{n-1}{2}\Big) \Big(\mu -\frac{n-1}{2}+1\Big)\frac1{r^2} \Psi_\mu^-
\] 
hence $\Psi_\mu^-$ is a spherical harmonics of degree $\mu - \frac{n-1}{2}+1$ if $\mu >0$ and $-\mu -\frac{n-1}{2}$ otherwise.

In other words
\[
\mathcal H_\mu \subseteq  (\mathcal S_{|\mu|-\frac{n-1}{2}} \oplus \mathcal S_{|\mu| - \frac{n-1}{2} + 1} ) \otimes \C^M.
\]
\end{proof}

\begin{proof}[Proof of Proposition \ref{prop:LPsetup}.] We have 
\[
\pi_j u = \sum_\mu \pi_j u_\mu
\]
where $u_\mu$ is the orthogonal projection of $u$ over $\mathcal H_\mu \otimes L^2_r$. If $|\mu| > b_j$, then
\[
|\mu| - \frac{n-1}2 > 2^{j+1}
\]
hence $u_\mu $ is a combination of spherical harmonics of degree higher than $2^{j+1}$ hence $\pi_j u_\mu =0$. 

If $|\mu| < a_j$ then 
\[
|\mu| - \frac{n-1}{2}+1 < 2^j
\]
hence $u_\mu$ is a combination of spherical harmonics of degree lesser than $2^j$, we have $\pi_j u_\mu = 0$ therefore 
\[
\pi_j u = \sum_{|\mu| \in [a_j,b_j]}\pi_j u_\mu = \pi_j p_{a_j,b_j} u.
\]
\end{proof}

\subsection{Proof of Theorem \ref{teo2}.} 

As done in \cite{cacdescurv1}, by relying on Littlewood-Paley theory on the sphere we are able to prove Strichartz estimates for the Dirac equation with general initial conditions in the setting of spherically symmetric manifolds. As the proof is very similar, we omit some details.

\begin{proposition}\label{prop:LPfirststep} Let $m,p,q$ be admissible. Let $a,b>0$ be such that
\[
\frac1{2a} + \frac5{pb} \leq 1
\]
if $m\neq 0$ or $n\neq 3$ and 
\[
\frac1{2a} + \frac5{pb}<1
\]
otherwise. We have for all $u_0 \in H^{a,b}(\R^n)$,
\[
\|S_n(t)u_0\|_{L^p(\R,W^{s,q}(\R^n))} \lesssim_{n,\varphi,m,p,q,a,b} \|u_0\|_{H^{a,b}}.
\]
\end{proposition}

\begin{proof} We have by the Littlewood-Paley theory ($q \in [2,\infty)$)
\[
\|S_n(t)u_0\|_{L^p(\R,W^{s,q}(\R^n))}^2 \lesssim \sum_j \|\pi_j S_n(t)u_0\|_{L^p(\R,W^{s,q}(\R^n))}^2 .
\]
 By Proposition \ref{prop:LPsetup}, we have 
\[
\|\pi_j S_n(t)u_0\|_{L^p(\R,W^{s,q}(\R^n))}^2  = \|\pi_j p_{a_j,b_j} S_n(t)u_0\|_{L^p(\R,W^{s,q}(\R^n))}^2 .
\]
Again by Littlewood-Paley theory, we have 
\[
\|\pi_j p_{a_j,b_j} S_n(t)u_0\|_{L^p(\R,W^{s,q}(\R^n))}^2\lesssim \|p_{a_j,b_j} S_n(t) u_0\|_{L^p(\R,W^{s,q}(\R^n))}^2.
\]

We apply Proposition \ref{prop:strichartrestricted}, we get
\[
\|\pi_j p_{a_j,b_j} S_n(t)u_0\|_{L^p(\R,W^{s,q}(\R^n))}^2\lesssim b_j^{10/p + \varepsilon} \|p_{a_j,b_j}u_0\|_{H^{1/2}}^2
\]
with $\varepsilon >0$ if $m=0$ and $n=3$ (and  $0$ otherwise). From the inequality
\[
xy \leq x^c + y^d
\]
for any $x,y\in [1,\infty)$ and $\frac1{c}+ \frac1{d} \leq 1$, we deduce
\[
\|\pi_j p_{a_j,b_j} S_n(t)u_0\|_{L^p(\R,W^{s,q}(\R^n))}^2\lesssim b_j^{(10/p + \varepsilon)d} \|p_{a_j,b_j}u_0\|_{L^2}^2 + \|p_{aj,b_j}u_0\|_{H^{c/2}}^2.
\]
Because $[a_j,b_j]$ is localized around $2^j$, we get
\[
\|S_n(t)u_0\|_{L^p(\R,W^{s,q}(\R^n))} \lesssim \|u_0\|_{H^{c/2, (5/p + \varepsilon/2)d}}.
\]
Setting $a = c/2$ and $b = (5/p + \varepsilon/2)d$, the condition on $c$ and $d$ becomes
\[
\frac1{2a} + \frac{5/p + \varepsilon/2}{b} \leq 1
\]
which is equivalent to the hypothesis of Proposition \ref{prop:LPfirststep} by discussing the possible values of $\varepsilon$.
\end{proof}

We now extend Lemmas \ref{lem:normHs} and \ref{lemsigma} to include the angular dependence. 

\begin{lemma}\label{lem:normHsbis} The multiplication by $\sigma_n$ is an isometry from $L^2(\R^n)$ to $L^2(\Sigma)$. The multiplication by $\sigma_n$ is an isomorphism from $H^1 (\R^n)$ to $H^1(\Sigma)$, the immediate consequence of which that for all $a\in [0,1]$, $b\in \R$, the multiplication by $\sigma_n$ is an isomorphism from $H^{a,b}(\R^n)$ to $H^{a,b}(\Sigma)$.
\end{lemma}

\begin{proof} The fact that the multiplication by $\sigma_n$ is an isometry from $L^2(\R^n)$ to $L^2(\Sigma)$ in already present in Lemma \ref{lem:normHs}.

We have for all $F \in \mathcal C^\infty (\Sigma, \C^M)$, writing 
\[
F = \begin{pmatrix} f \\ g \end{pmatrix}
\]
with $f$ and $g$ in $\mathcal C^\infty (\Sigma, \C^{M/2})$
\[
h^{ij}\an{D_i F, D_j F}_{\C^M} = \an{\partial_r F ,\partial_r F}_{\C^M} + \frac1{\varphi^2} \Big( \tilde h^{ij} \an{\mathbb D^\varphi_i f, \mathbb D_j^\varphi f}_{\C^{M/2}} +  \tilde h^{ij} \an{\mathbb D_i^\varphi g, \mathbb D_j^\varphi g}_{\C^{M/2}}\Big) 
\]
where $\mathbb D^\varphi_j = \tilde D_j + 2i \varphi' \tilde e^a_{\; j} \tilde \Sigma_{0,a}$ with $\tilde D$  the covariant derivatives for spinors on the sphere and $\tilde h$ is the metric of the sphere. 

The fact that
\[
\|\an{\partial_r (\sigma_n F) ,\partial_r (\sigma_n F)}_{\C^M}\|_{L^1(\Sigma)} \sim \|\partial_r F\|_{L^2(\R^n)}^2
\]
is due to Lemma \ref{lem:normHs}.

We have 
\[
\mathbb D^\varphi_j = \mathbb D^r_j + 2i (\varphi' -1) \tilde e^a_{\; j}\tilde \Sigma_{0,a}.
\]
Thanks to the Cauchy-Schwarz inequality applied to the scalar product $x,y\mapsto h^{ij}x_iy_j$ we get
\begin{equation}\label{comparingDerivatives}
\sqrt{\frac1{\varphi^2}\tilde h^{ij} \an{\mathbb D^\varphi_i ,\mathbb D^\varphi_j f}_{\C^{M/2}}} \lesssim \frac{r}{\varphi} \sqrt{\frac1{r^2}\tilde h^{ij} \an{\mathbb D^r_i ,\mathbb D^r_j f}_{\C^{M/2}}} + \frac{|\varphi' -1|}{\varphi} \sqrt{\an{f,f}_{\C^{M/2}}}
\end{equation}
and conversely
\[
\sqrt{\frac1{r^2}\tilde h^{ij} \an{\mathbb D^r_i ,\mathbb D^r_j f}_{\C^{M/2}} }\lesssim \frac{\varphi}{r} \sqrt{\frac1{\varphi^2}\tilde h^{ij} \an{\mathbb D^\varphi_i ,\mathbb D^\varphi_j f}_{\C^{M/2}}} + \frac{|\varphi' -1|}{r} \sqrt{\an{f,f}_{\C^{M/2}}}
\]

Because $\mathbb D$ and $\sigma_n$ commute, we get
\[
\mathbb D^\varphi_j (\sigma_n f) = \sigma_n \mathbb D^r_j f + \sigma_n 2i (\varphi'-1)\tilde e^a_{\; j}\tilde \Sigma_{0,a}.
\]
To ensure that 
\[
\big\| \frac1{\varphi^2}\tilde h^{ij} \an{\mathbb D^\varphi_i \sigma_n f, \mathbb D_j^\varphi \sigma_n f}_{\C^{M/2}}\big\|_{L^1(\Sigma)}
\]
it is thus sufficient to prove that $ \frac{\varphi}{r}$, $\frac{r}{\varphi}$ and $\frac{\varphi'-1}{\varphi}$ are bounded. But $\varphi = r(1+\varphi_1)$ with $\varphi_1$ non negative, bounded, a $O(r)$ in $0$ and thus that $\varphi_1'$ is bounded, hence
\[
\frac{\varphi}{r} = 1+\varphi_1,\quad \frac{r}{\varphi} = \frac1{1+\varphi_1},\quad \frac{\varphi'-1}{\varphi} = \frac{\varphi_1}{r(1+\varphi_1)} + \frac{\varphi_1'}{1+\varphi_1}
\]
are bounded.
\end{proof}

\begin{lemma}\label{lem:LpNormBis} The multiplication by $\sigma_n$ is a continuous operator from $L^p(\R, W^{s,q}(\R^n))$ to $\sigma_n^{1-2/q} L^p(\R,W^{s,q}(\Sigma))$ for any $p\in [1,\infty]$, $q\in (1,\infty)$, $s\in [-1,1]$.
\end{lemma}

\begin{proof}
As in Lemma \ref{lemsigma}, we reduce our proof to the proof of, for all $q\in (1,\infty)$, \begin{enumerate}
\item the multiplication by $\sigma_n^{2/q}$ is an isometry from $L^q(\R^n)$ to $L^q (\Sigma)$,
\item the multiplication by $\sigma_n^{2/q}$ is continuous from $W^{1,q}(\R^n)$ to $W^{1,q}(\Sigma)$,
\item the multiplication by $\sigma_n^{-2/q}$ is continuous from $W^{1,q}(\Sigma)$ to $W^{1,q}(\R^n)$.
\end{enumerate}

(1) The proof of (1) is similar to what we have already done in the proof of Lemma \ref{lemsigma}.

(2) With the same notations as in the proof of Lemma \ref{lem:normHsbis}, and keeping in mind (2) in Lemma \ref{lemsigma}, it remains to prove (with a slight abuse of notation) that for all $f\in W^{1,q}(\R^n)$, 
\[
\big\| \sqrt{\frac1{\varphi^2} \tilde h^{ij} \an{\mathbb D_i^\varphi (\sigma_n^{2/q} f), \mathbb D_j^\varphi (\sigma_n^{2/q} f)}_{\C^{M/2}}}\big\|_{L^q(\Sigma)} \lesssim \|f\|_{W^{1,q}(\R^n)}.
\]
But because of (1) and the fact that $\sigma_n^{2/q}$ and $\mathbb D^\varphi$ commute, it sufficient to prove that
\[
\big\| \sqrt{\frac1{\varphi^2} \tilde h^{ij} \an{\mathbb D_i^\varphi f, \mathbb D_j^\varphi f}_{\C^{M/2}}}\big\|_{L^q(\R^n)} \lesssim \|f\|_{W^{1,q}(\R^n)}.
\]
We now use the inequality \eqref{comparingDerivatives} and the fact that $\frac{r}{\varphi}$ and $\frac{\varphi' -1}{\varphi}$ are bounded to get the result.

(3) Similar to (2).
\end{proof}

Therefore, combining Proposition \ref{prop:LPfirststep} with Lemma \ref{lem:LpNormBis} eventually yields the Proof of Theorem \ref{teo2}.


\appendix
\section{Comments on admissible manifolds}

It is natural to ask wether conditions \eqref{crucialcond} are fullfilled by other natural choices of the function $\varphi(r$), as e.g. ${\varphi(r)}=\sinh(r)$ (which corresponds to {\em hyperbolic spaces}), or ${\varphi}(r)=r+r^2+\dots+ r^p$ with $p>2$ (manifolds with polynomial growth). It turns out that with both these choices conditions \eqref{crucialcond} are only satisfied for large $r$; more precisely, the following result holds

\begin{proposition}\label{verification2} 
Let $(\M,g)$ defined by $\M=\mathbb{R}_t\times \Sigma$, with $(\Sigma,\sigma)$ a warped product manifold with metric given by \eqref{warped}, and let ${\varphi(r)}=\sinh(r)$ or ${\varphi}(r)=r+r^2+\dots+ r^p$ with $p\in\N$ and $p>2$. Then, condition \eqref{crucialcond} is not satisfied.

\end{proposition}

\begin{proof}

It is quite immediate to see that condition
$$
4r^2{V}_{\mu} + 1>0\Leftrightarrow 4r^2\mu^2\pm 4r^2\mu \cosh(r)+\sinh(r)^2>0,
$$
is true for any $\mu$ if and only if 
$$
\left(\frac{\varphi'(r)}{\varphi(r)}\right)^2<\frac1{r^2},
$$
and this last condition is not satisfied by the choices $\varphi(r)=\sinh(r)$ or $\varphi(r)=r+r^2+\dots+r^p$.  We omit the details.
\end{proof}

\begin{remark}
As a matter of fact, it might be possible to prove that with the choices of $\varphi(r)$ of Proposition \ref{verification2},  condition \eqref{crucialcond} is actually satisfied for $r$ larger than a sufficiently large $R=R(\mu)$; as a consequence, it would be tempting to consider manifolds that are flat inside some balls, and then present different asymptotic behaviors (like, for instance, asymptotically hyperbolic manifolds). These cases would correspond to choosing a function $\varphi(r)\in C^\infty(\R^+)$ that takes the form
\begin{equation}\label{structphi}
\varphi(r)=
\begin{cases}
r \qquad {\rm if}\: r\leq R,\\
\psi(r)\quad {\rm if}\: R\leq r\leq 2R,\\
\sinh(r)\quad {\rm if}\: r> 2R,
\end{cases}
\end{equation}
(and analogous in the case of manifolds with polynomial growth). The existence of such a function is quite standard; on the other hand, we are not able to show that condition \eqref{crucialcond} is satisfied everywhere. In any case, the fact that the quantity $R$ will depend on $\mu$ makes the analysis in these cases not so relevant from a geometrical point of view, and therefore we prefer to leave the study of these other geometries to future investigations.
\end{remark}


\begin{thebibliography}{9}


\bibitem{bar}
C. B\"ar.
\newblock The Dirac operator on hyperbolic manifolds of finite volume.
\newblock {\em J. Differential Geometry}, 53 439-488 (1999).



\bibitem{banduy}
V. Banica and T. Duyckaerts.
\newblock Weighted Strichartz estimates for radial Schr\"odinger equation on noncompact manifolds. \newblock {\em Dyn. Partial Differ. Equ.} 4, no. 4, 335-359 (2007). 

\bibitem{kg}
J. Ben-Artzi, F. Cacciafesta, A.S. de Suzzoni and J. Zhang.
\newblock Strichartz estimates for Klein-Gordon equation in a conical singular space,
\newblock preprint, \url{https://arxiv.org/abs/2007.05331}.
%
%
%
%
%
%
\bibitem{cacdes1}
F. Cacciafesta and A.S. de Suzzoni.
\newblock Weak dispersion for the Dirac equation on asymptotically flat and warped products spaces, 
\newblock {\em Discrete Contin. Dyn. Syst} 39(8), pp. 4359-4398 (2019).
%
%
%
\bibitem{cacdescurv1}
F. Cacciafesta,  A.S. de Suzzoni.
\newblock local-in-time Strichartz estimates for the Dirac equation on spherically symmetric spaces.
\newblock preprint, \url{https://arxiv.org/abs/1902.07572}.

%
\bibitem{camporesi}
R. Camporesi and A Higuchi.
\newblock On the eigenfunctions of the Dirac operator on spheres
and real hyperbolic spaces.
\newblock {\em Journal of Geometry and Physics} 20 l-1 8 (1996).


\bibitem{chou} A.W. Chou, 
\newblock The Dirac operator on spaces with conical singularities and positive scalar curvatures.
\newblock {\em Transactions of AMS}, 289, 1-40 (1985).


\bibitem{dai}
F. Dai and Y. Xu,
\newblock Approximation theory and harmonic analysis on spheres and balls,
\newblock {\em Springer Monographs in Mathematics.}


\bibitem{dansmoot}
P. D'Ancona. 
\newblock Kato smoothing and Strichartz estimates for wave equations with magnetic potentials.
\newblock {\em Comm. Math. Phys.} 335 1-16, (2015).

\bibitem{danzha}
P. D'Ancona and Q. Zhang.
\newblock Global existence of small equivariant wave maps on rotationally symmetric manifolds. 
\newblock {\em Int. Math. Res. Not. IMRN}, no. 4, 978-1025  (2016).
%
\bibitem{daude1}
T. Daud\'e, D. Gobin, and F. Nicoleau.
\newblock Local inverse scattering at fixed energy in spherically symmetric asymptotically hyperbolic manifolds.
\newblock {\em Inverse Probl. Imaging }10, no. 3, 659-688  (2016).
%
%
%
%
%
%
%

\bibitem{spinstructure}
A. Haefliger.
\newblock Sur l'extension du groupe structural d'une espace fibr\'{e}.
\newblock {\em C. R. Acad. Sci. Paris}, 243:558--560, (1956).

\bibitem{kato1}
T. Kato. 
\newblock Wave operators and similarity for some non-selfadjoint operators. 
\newblock {\em Math. Ann.}, 162 258- 279, (1965/1966).

\bibitem{kato2}
T. Kato and K. Yajima. 
\newblock Some examples of smooth operators and the associated smoothing effect. 
\newblock {\em Rev. Math. Phys.}, 1(4):481-496 (1989).

%
%
\bibitem{mach}
Shuji Machihara, Makoto Nakamura, Kenji Nakanishi, and Tohru Ozawa.
\newblock Endpoint {S}trichartz estimates and global solutions for the
  nonlinear {D}irac equation.
\newblock {\em J. Funct. Anal.}, 219(1):1--20 (2005).
%
%
%
%

\bibitem{parktoms}
L. E. Parker and D. J. Toms,
\newblock Quantum field theory in curved spacetime.
\newblock Cambridge university press.

\bibitem{rodsch}
I. Rodnianski and W. Schlag. 
\newblock Time decay for solutions of Schr\"odinger equations with rough and time-dependent potentials. 
\newblock {\em Invent. Math.}, 155(3):451-513 (2004).

\bibitem{roe}
J. Roe.
\newblock Elliptic operators, topology and asymptotic methods,
\newblock Second Edition Chapman and Hall/CRC.


\bibitem{thaller}
\newblock B. Thaller,
\newblock {\em The Dirac Equation},
\newblock {Springer-Verlag}, Texts and Monographs in Physics (1992).

\end{thebibliography}
\end{document}